\documentclass[10pt,reqno]{amsart}
\usepackage{networkdesign}
\usepackage{natbib}

\title{Network Design with Probabilistic Capacities}
\author{\small{Alper Atamt\"urk and Avinash Bhardwaj}}

\thanks{ \noindent \hskip -1mm
A. Atamt\"urk: Department of Industrial Engineering \& Operations Research, University of California, Berkeley, CA 94720.
\texttt{atamturk@berkeley.edu}   \\
A. Bhardwaj: Department of Industrial Engineering \& Operations Research, University of California, Berkeley, CA 94720. \texttt{avinash@ieor.berkeley.edu}.
Currently: CORE, Université catholique de Louvain, Voie du Roman Pays 34, B-1348 Louvain-la-Neuve, Belgium. \texttt{avinash.bhardwaj@uclouvain.be}.
}

\begin{document}

\maketitle

\begin{abstract}
We consider a network design problem with random arc capacities and give a formulation with a probabilistic capacity constraint on each cut of the network. To handle the exponentially-many probabilistic constraints a separation procedure that solves a nonlinear minimum cut problem is introduced. For the case with independent arc capacities, we exploit the supermodularity of the set function defining the constraints and generate cutting planes based on the supermodular covering knapsack polytope. For the general correlated case, we give a reformulation of the constraints that allows to uncover and utilize the submodularity of a related function. The computational results indicate that exploiting the underlying submodularity and supermodularity arising with the probabilistic constraints provides significant advantages over the classical approaches. \\

\end{abstract}

\begin{center} August 2016; April 2017 \end{center}

\section{Introduction}

\label{sec:intro}

Designing networks with random components often necessitates building costly extra capacity to satisfy a desired service level.
Examples include communication networks, air and ground traffic networks, supply-chain networks.
Uncertainties in network capacities  may be caused by a variety of natural or artificial factors.
In the context of air or ground transportation network infrastructure, the capacity uncertainties are typically due to weather conditions and traffic along with other factors \cite{Chen:survey}. Similar capacity uncertainties arise in the design and operation of communication network infrastructure, power grid, and related applications \citep{classen11,Kennington:2010,koster2013robust,YW:TEP}. It is of significant interest to
policy makers and other stakeholders to understand the trade-off between the network design cost and the service level to assess
the value of additional capacity investments. To this end, we give an optimal network design model with probabilistic capacity constraints and propose effective solution methods that exploit the underlying combinatorial properties.

Let ${N}({V}, {A})$ be a network with vertices $V$ and arcs $A$ with $m = |A|$.
 Given random arc capacities $\xi_a, \ a \in A$, we are interested in the problem of selecting a minimum-cost
 subset of the arcs to satisfy the demand on the network with a high probability.
 Without loss of generality, we may assume that there is a single source $s$ and a single sink $t$ with demand $d$, as
 one can collate sources and sinks into a single node each respectively via arcs with deterministic bounds.
 Let $h_{a}$  be the fixed ``design" cost of arc  $a \in {A}$ for inclusion in the selected sub-network.

Since the maximum flow on the network is limited by the capacity of a minimum cut, we need to ensure sufficient capacity on each cut of the network.	
The problem is modeled as finding a minimum-cost subset of the arcs so that all cuts of the network jointly have sufficient capacity to meet the demand with probability at least  $1 - \epsilon$, where $0 < \epsilon \le 0.5$. Letting $x_{a}$ be 1 if arc $a$ is selected, 0 otherwise, the probabilistic network design problem is formulated as:
\begin{align}
\label{eq:chanceconstrained}
\JPND  \  \  \min_{\x \inset \binset^m} \ \textbf{h}\,' \x \text{ : }  \text{\textbf{Pr}} \left(\sum_{a \in {C}}\,{\xi_{a} x_{a}} \geq d, \,\,\forall\,C \inset \mathcal{C} \right) \geq 1-\epsilon,
\end{align}
where $\mathcal{C}$ denotes the set of all $s-t$ cuts in the network ${N}(V,A)$. The constraint set with joint probabilistic constraints across all the cuts of the network
is non-convex and not well-understood even for the continuous relaxation of the problem \cite{prekopa-book,SDR:lectures}. Therefore, for efficient computational methods,
one often resorts to convex approximations of the joint probabilistic constraints \cite{NS:convex,A:convex}.

In this paper, we study an approximate model with a disjoint probabilistic capacity constraint for each cut of the network.
That is, we look for a minimum cost subset of the arcs so that each cut of the selected sub-network satisfies the demand with probability at least  $1 - \epsilon$:
\begin{align}
\label{eq:chanceconstrained}
\PND  \  \  \min_{\x \inset \binset^m} \ \textbf{h}\,' \x \text{ : }  \text{\textbf{Pr}} \left(\sum_{a \in {C}}\,{\xi_{a} x_{a}} \geq d\right) \geq 1-\epsilon, \,\,\,\forall\,C \inset \mathcal{C},
\end{align}
where $\mathcal{C}$ denotes the set of all $s-t$ cuts in the network ${N}(V,A)$.

If $\bs\xi$ is normally distributed with mean $\bs{\mu}$ and covariance $\bs\Sigma$, problem \eqref{eq:chanceconstrained} is equivalently stated with conic quadratic constraints:
\begin{align}
\label{eq:equivalence}
\CQND  \ \  \ \min_{\x \inset \binset^m} \ \textbf{h}\,' \x \text{ : }  {\bs \mu}_C' \x - \Omega\norm{\bs{\Sigma}_C^{1/2}\,\x}\, \geq\, d,
     \,\forall\,\,C \inset \mathcal{C},
\end{align}
where $\Omega = \phi^{-1}(1-\epsilon)$ and $\phi$ is the standard normal c.d.f.,
$\bs{\mu}_{{C}}$ is an $m$-vector of mean capacities for the arcs in cut ${C}$, and zero otherwise, and $\boldsymbol{\Sigma}_{{C}}$ is an $m \times m$-matrix of covariances for the arcs in ${C}$ and zero otherwise. 
The term $\Omega\norm{\bs{\Sigma}^{1/2}\,\x}$ is used to build sufficient slack into the constraint to accommodate the variability of $\bs \xi$ around its mean $\bs \mu$.
Problem \eqref{eq:equivalence} is \NP-hard as the the deterministic network design problem with $\Omega = 0$ is \NP-hard.
We use $\sigma_{ij}$ to denote the $(ij)$th coefficient of the covariance matrix $\bs \Sigma$ and $\sigma_i^2$ for the $i$th variance ($\sigma_{ii})$.

If the distribution of $\bs\xi$ is unknown, conservative robust versions can be modeled as \eqref{eq:equivalence} by using appropriate choices of $\Omega$.  For example, if $\xi_i$'s are known only through their first two moments, then inequalities in \eqref{eq:equivalence} with $\Omega = \sqrt{(1-\epsilon)/\epsilon}$ imply the probabilistic constraints in \eqref{eq:chanceconstrained} \cite{Bertsimas2005,Ghaoui2003}. Alternatively, if $\xi_i$'s are independent and symmetric with support $[u_i - \sigma_i, u_i + \sigma_i]$, then inequalities in \eqref{eq:equivalence} with $\Omega = \sqrt{\ln(1/\epsilon)}$ and diagonal $\bs \Sigma = \bs D$, where $D_{ii} = \sigma_i^2$, imply the probabilistic constraints in \eqref{eq:chanceconstrained} \cite{Ben-Tal2000,Ben-Tal2002}. Hence, under different assumptions on $\boldsymbol{\xi}$, one arrives at different instances of the model \eqref{eq:equivalence} with conic quadratic constraints ($\Omega \ge 0$).

\textit{Contributions}.
There are two main challenges in solving \CQND. The first one is that the formulation has an exponentially-many conic quadratic constraints --- one for each cut of the network. We use a separation approach to generate only a small subset of the conic quadratic inequalities. Finding a violated conic quadratic constraint \eqref{eq:equivalence} is modeled as a minimum cut problem with a concave objective, which includes the \NP-hard max-cut problem as a special case. We give alternative formulations and compare them for the independent and correlated cases.
The second challenge is that there is often a large gap between the optimal objective and its convex relaxation. In order to strengthen the convex relaxations, we derive strong cutting planes that make use of submodularity. In particular, we
show that when the capacities are independent, under a mild assumption, the feasible
set of \CQND \ is the intersection of monotone supermodular covering knapsacks. For the general correlated case, we give a reformulation to uncover the combinatorial structure of another set function for the underlying constraints. This reformulation allows us to make use of submodularity
even if the original constraint function is not submodular. The proposed reformulation is independent from the network structure and, therefore, it can be useful for other combinatorial problems with a risk objective or constraint with correlated random variables.

\textit{Literature review}.
By now there is large body of literature on the polyhedral investigations of deterministic capacitated network design \citep[e.g.,][]{ANS:avub,AG:mixset,A:nd,A:fp,BG:cnd,HKLS:edge-cap,MM:st-path,MMV:conv-2core,AG:nd-review}. While the capacity uncertainty is largely studied
	in the context of survivable network design that can stand any single link failure 
	\cite[e.g.,][]{BM:surv-cuts,BMM:survconn,BMM:conn-split,GMS:ndcrev,SD:snd,RA:surv-colgen}, the treatment of the variability in the capacities is limited.
Link capacity variability is well documented in transportation networks. \citet{Chen:reliabity02} study the capacity reliability of a road network, that is, the probability that a network can support a given demand at a given service level when the link capacities are uncertain. \citet{LT:netdegrade03} formulate a probabilistic user equilibrium model for maximizing the network flow capacity while accounting for the effects of travel time reliability due to degradable links in terms of predetermined link performance function (proxy for capacity distributions). We refer the reader to \citet{Chen:survey} for a survey on the stochastic transportation network design problem. In the context of communication networks, \citet{classen11} study a broadband wireless network which uses microwave links to deliver data, where the channel conditions account for uncertainty in link capacities. \citet{YW:TEP} consider the transmission expansion problem in a power network for handling uncertainties corresponding to future plant locations in terms of generation expansion and load growth. \citet{Fortz:RobustLinear} study the commodity packing problem on a network with uncertain capacities and linearize the probabilistic linear constraint \eqref{eq:equivalence} when $\xi_a$'s are independent and $\xi_a \sim \mathcal{N}(\mu_a,\lambda\mu_a)\,\,\,\forall\,a \in A$, i.e., constant coefficient of variation. \citet{SLK:joint-chance} give a scenario-based approach to multi-row probabilistic constraints on binary variables. We do not address joint probabilistic constraints or multiple commodities here. Cutting planes for general conic quadratic mixed 0-1 programs are given in \cite{Cezik2005,AN:conicmir,kilincc2014minimal,kilincc2014two}.

\textit{Example}.
We end this section with a motivating example that highlights the value of model \CQND \ in constructing a trade-off curve of service level versus design cost. To do so we construct networks with varying service levels and design costs; and with simulations compute the observed service levels. The data for the sample six-node network is presented in Table \ref{tab:1} of Appendix~\ref{app:data}.

\begin{figure}[b]
	\centering
	\begin{subfigure}{0.3\textwidth}
	\centering
	\begin{tikzpicture}[scale=0.8, auto,swap]
	    \foreach \pos/\name in {{(2,1)/1}, {(4,1)/2},
	                            {(1,0)/s}, {(5,0)/t}, {(2,-1)/3}, {(4,-1)/4}}
	        \node[vertex] (\name) at \pos {$\name$};
	    \foreach \source/ \dest in {s/2,s/4,s/t,2/t,4/t}
	        \path[selected edge] (\source)--(\dest);
	
	    \foreach \source/ \dest in {s/1,s/2,s/3, s/4,s/t,1/2,1/3,1/4,1/t,2/3,2/4,2/t,3/4,3/t,4/t}
	        \path[ignored edge] (\source)--(\dest);
	\end{tikzpicture}
	\caption{service level 50\%}
	\label{fig:1a}
	\end{subfigure}%
	~
	\begin{subfigure}{0.3\textwidth}
	\centering
	\begin{tikzpicture}[scale=0.8, auto,swap]
	    \foreach \pos/\name in {{(2,1)/1}, {(4,1)/2},
	                            {(1,0)/s}, {(5,0)/t}, {(2,-1)/3}, {(4,-1)/4}}
	        \node[vertex] (\name) at \pos {$\name$};
	    \foreach \source/ \dest in {s/1,s/2,s/4,s/t,1/3,2/t,3/t,4/t}
	        \path[selected edge] (\source)--(\dest);
	
	    \foreach \source/ \dest in {s/1,s/2,s/3, s/4,s/t,1/2,1/3,1/4,1/t,2/3,2/4,2/t,3/4,3/t,4/t}
	        \path[ignored edge] (\source)--(\dest);
	\end{tikzpicture}
	\caption{service level 80\%}
	\label{fig:1c}
	\end{subfigure}%
~
	\begin{subfigure}{0.3\textwidth}
	\centering
	\begin{tikzpicture}[scale=0.75, auto,swap]
	    \foreach \pos/\name in {{(2,1)/1}, {(4,1)/2},
	                            {(1,0)/s}, {(5,0)/t}, {(2,-1)/3}, {(4,-1)/4}}
	        \node[vertex] (\name) at \pos {$\name$};
	    \foreach \source/ \dest in {s/1,s/2,s/4,s/t,1/t,2/t,4/t}
	        \path[selected edge] (\source)--(\dest);
	
	    \foreach \source/ \dest in {s/1,s/2,s/3, s/4,s/t,1/2,1/3,1/4,1/t,2/3,2/4,2/t,3/4,3/t,4/t}
	        \path[ignored edge] (\source)--(\dest);
	\end{tikzpicture}
	\caption{service level 97.5\%}
	\label{fig:1c}
	\end{subfigure}
	\caption{Minimum cost networks corresponding to different service levels.}
	\label{fig:1}
\end{figure}
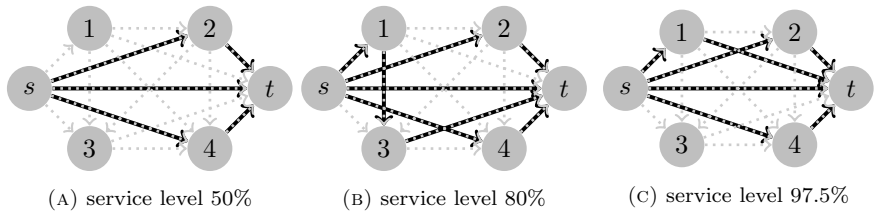

First using reformulation \CQND \ we construct minimum cost networks for six service levels: $50\%$, $70\%$, $80\%$, $97.5\%$, $99\%$, $99.9\%$ . The $50\%$ service level is the nominal case corresponding to the deterministic network design with $\Omega=0$. The network configurations obtained for $50\%, 80\%$ and $97.5\%$ service levels are shown in Figure~\ref{fig:1}.
Note that the networks are not nested, that is, the arcs used for lower service levels are not necessarily included in the networks with higher service levels. The complete set of network configurations for all service levels are displayed in Appendix~\ref{app:conf}.

Next for each network configuration, we generate 10,000 samples from the corresponding arc capacity distributions.
For each network sample we check whether that network satisfies the demand, i.e., whether the minimum cut capacity is greater than or equal to the demand to compute the simulated service level. The results of this simulation are summarized in Table \ref{tab:2}.
The frequency distributions shown in Appendix~\ref{app:dist} indicate that they are skewed and non-normal for all network configurations.
Nevertheless, the simulated service levels are reasonably close to the model service levels and increase monotonically with $1-\epsilon$.

\begin{table}[h!]
\caption{Simulation results.}
\centering
	\begin{tabular}{c|c|c|c|c|c}
		\hline \hline
		\multirow{2}{*}{$1-\epsilon$} & \multirow{2}{*}{cost} & \multicolumn{3}{c|}{minimum cut} & simulated \\\cline{3-5}
		& & min & mean & max & service level \\ \hline
		50\%   & 100\% & 121.9 & 222.1 & 274.6 & 39.81\% \\[1mm]
		70\%   & 104\% & 124.0 & 238.4 & 293.2 & 70.44\% \\[1mm]
		80\%   & 127\% & 154.0 & 249.2 & 306.5 & 82.68\% \\[1mm]
		97.5\% & 135\% & 195.2 & 301.4 & 356.4 & 99.68\% \\[1mm]
		99\%   & 177\% & 209.6 & 303.6 & 360.2 & 99.85\% \\[1mm]
		99.9\% & 186\% & 209.6 & 313.4 & 368.5 & 99.96\% \\
		\hline \hline 
	\end{tabular}
\label{tab:2}
\end{table}

Figure~\ref{fig:tradeoff} shows the trade-off between the cost of the network design and the service level.
The red curve is for the predicted model service level, whereas the blue curve is for the simulated service level.
As expected, the cost of the design increases with the service level, at steeper rates for higher service levels.
The trade-off curve is not convex as the network configuration changes in discrete steps.

\begin{figure}[h!]
\centering
\includegraphics[scale = 0.15]{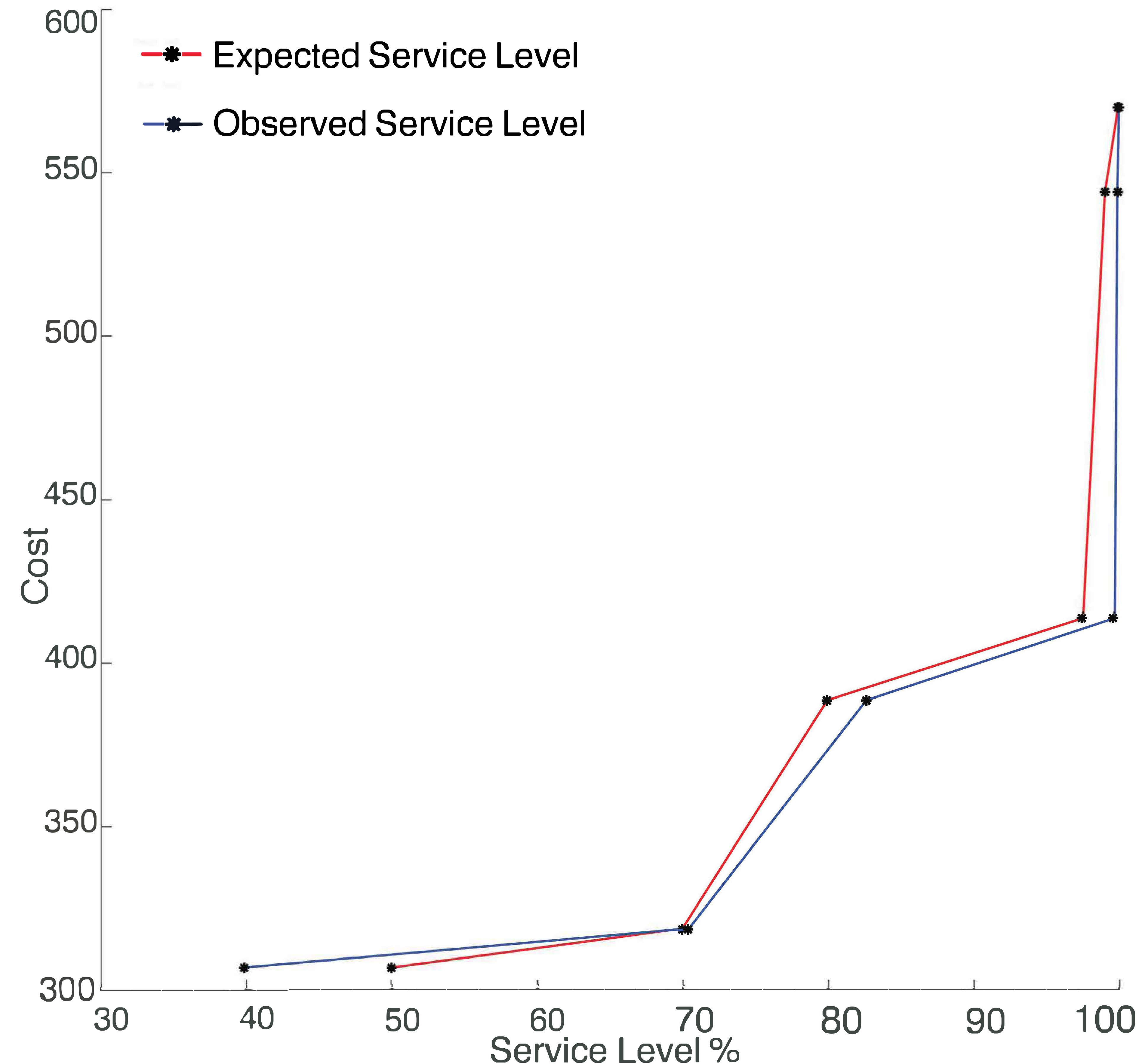}	
\caption{Network cost versus service level, model versus simulated.}
\label{fig:tradeoff}
\end{figure}

The remainder of this paper is organized as follows:
In Section~\ref{sec:strong} we give reformulations and strong valid inequalities for the independent and correlated capacities, separately.
In Section \ref{sec:Separation} we discuss the separation problem for the capacity constraints for each case and compare alternative formulations. In Section \ref{sec:Computations} we present computational experiments on testing the effectiveness of the valid inequalities in solving the probabilistic network design problem.

\section{Strengthening the formulation}
\label{sec:strong}

In this section we describe how to strengthen the convex relaxation of model \CQND. We consider each constraint in \eqref{eq:equivalence} separately and, for simplicity of notation, drop the subscript $C$.
The case with independent arc capacities warrants special consideration as it leads to a specific combinatorial structure. Therefore, we discuss the cases of independent capacities and correlated capacities  separately in Sections \ref{sec:indepMaster} and \ref{sec:corrMaster}.

\subsection{Independent arc capacities}
\label{sec:indepMaster}
When the arc capacities of the network are independent,
the covariance matrix $\bs{\Sigma}$ reduces to a diagonal matrix $\bs D$ of variances; that is, $D_{ii} = \sigma^2_i$.
Then, as $x_i^2 = x_i$ for binary $\x$, each constraint in \eqref{eq:equivalence} reduces to
\begin{align*}
 f(\bs x) =  \bs{\mu}'\x - \Omega\sqrt{\bs{D\,\x}}\, \geq\, d.
\end{align*}

Given a finite ground set $A$, a set function $g\,:\,2^A\,\rightarrow {\R}$ is supermodular if
its difference function
\begin{align*}
    \rho_i(S) := g(S\cup i) - g(S) \,\,\,\text{for } S \subseteq N\setminus i
\end{align*}
is non-decreasing on $N\setminus i$; that is,
$\rho_i(S) \leq \rho_i(T), \,\forall\,S\subseteq T\subseteq N\setminus i$ and $i \in N$ \cite{Schrijver2003}.
If $f$ is supermodular, then $-f$ is submodular and vice-versa.

In a recent paper, \citet{ab:packs} study the polyhedral structure of the supermodular covering knapsack set
\begin{align*}
	K_g \define \set{\x \in \binset^{n}}{g(\x) \geq d}
\end{align*}
for a non-decreasing supermodular set function $g$. They give strong valid inequalities, referred to as the pack inequalities and the extended pack inequalities. Let $P \subseteq N$ be a maximal set satisfying $g(P) < d$. For such a $P$ it is easy to see that the corresponding pack inequality
\begin{align} \label{eq:pack-ineq}
	x(N \setminus P) \ge 1
\end{align}
is valid. Extended pack inequalities generalize \eqref{eq:pack-ineq} through sequential lifting.

It is easy to check that  $f(\bs x) =  \bs{\mu}'\x - \Omega\sqrt{\bs{D\,\x}}$ is supermodular on $\{0,1\}^A$ \cite[e.g.][]{Ahmed2011}. Moreover, it is reasonable to assume that the network capacity increases with the addition of new arcs. 
In particular, $f$ is non-decreasing if
    \begin{enumerate}[label = (CV)] 
    	\item \hspace{4cm} $\mu_a \ge \Omega \sigma_a$ for all $a \in A$. \label{asmptn:coeffofvar}
    \end{enumerate}

For the majority of the paper we will assume that \ref{asmptn:coeffofvar} holds.
In order to strengthen the continuous relaxation of \CQND for the independent case,
we utilize the extended pack inequalities of \citet{ab:packs} for $K_f$. The computational effectiveness of the pack inequalities and their extensions for the probabilistic network design problem is detailed in Section \ref{sec:Computations}.
In Section~\ref{sec:relax-assumption}, we show that the results can be applied and are useful even if assumption \ref{asmptn:coeffofvar} does not hold for all arcs.

Unfortunately supermodularity is lost once correlations are introduced to the covariance matrix $\bs \Sigma$.  In the following section we show how to exploit additional combinatorial structure when taking into account the correlations among the arc capacities.

\subsection{Correlated arc capacities}
\label{sec:corrMaster}

For a non-diagonal covariance matrix $\bs \Sigma$, the probabilistic cut capacity function
\begin{align*}
 f(\bs x) =  \bs{\mu}'\x - \Omega\sqrt{\x' \bs{\Sigma\,\x}}
\end{align*}
is not supermodular and the pack inequalities and their extensions that are valid for the diagonal case are no longer applicable.
Nevertheless, we show below that under assumption \ref{asmptn:coeffofvar}, we can utilize submodularity of a related function instead of the supermodularity of the original cut capacity function.

First, observe that, since $\Omega \ge 0$,  $f(\x) \geq d$ implies $\bs \mu'\x \geq d$.
Then, for $\x : \bs \mu'\x \geq d$, we have
\[
f(\x) \geq d \  \Longleftrightarrow  \ \Omega^2 \x' \bs \Sigma \x \le (\bs \mu' \x - d)^2.
\]
Therefore, for $\x : \bs \mu'\x \geq d$,
instead of the original conic quadratic constraint $f(\x) \ge d$,
we may use the equivalent quadratic constraint
\begin{align} \label{eq:q}
q(\x):= \Omega^2 \x' \bs \Sigma \x - (\bs \mu' \x - d)^2 \le 0,
\end{align}
which allows one to exploit submodularity. The function $q$ is the difference of two convex functions; it is neither convex, not concave.
However, as shown below, $q$ is submodular under assumption \ref{asmptn:coeffofvar}.

\begin{prop}
  \label{prop:Qsubmod}
  \citep{fishernemhausewolsey}
  Let $p: \binset^n \rightarrow \R$ be the quadratic function
  \begin{equation*}
    p(\x) \define \x'\,{\bf{Q}}\,\x,
  \end{equation*}
  where ${\bf{Q}}$ is a symmetric matrix. $p$ is submodular if $Q_{ij} \leq 0, \ 1\leq \,i < j\,\leq n.$
\end{prop}

\begin{prop}
  \label{prop:q_submod} If \ref{asmptn:coeffofvar} holds, then
  the set function $q(\x) =  \Omega^2 \x' \bs \Sigma \x - (\bs \mu' \x - d)^2$ is submodular.
  \end{prop}
\begin{proof}
Note that by observing $x_i = x_i^2$ for binary $\x$, $q(\x)$ can be written as
\begin{align}
\label{fn:fd}
	q(\x) = \sum_{i=1}^n \alpha_ix_i + 2\sum_{i=1}^n\sum_{j>i}^n \beta_{ij}x_ix_j - d^2,
\end{align}
where $\alpha_i = \Omega^2\sigma_i^2 + 2\mu_id - \mu_i^2$ and $\beta_{ij} = \Omega^2\sigma_{ij} - \mu_i\mu_j$.
 As $\sigma_{ij} \le \sigma_i \sigma_j$ and by assumption \ref{asmptn:coeffofvar}, we have
\[
\beta_{ij} = \Omega^2\sigma_{ij} - \mu_i\mu_j \leq \Omega^2\sigma_{i}\sigma_j - \mu_i\mu_j \leq 0.
\]
Then from Proposition~\ref{prop:Qsubmod} it follows that $q$ is submodular.
\end{proof}

Submodular pack \citep{ab:packs} and cover \cite{Atamturk2009333} inequalities and their extensions have been used to strengthen the continuous relaxations of submodular knapsack sets when the underlying set function is monotone. Unfortunately, $q$ is neither non-increasing nor non-decreasing and, therefore, these inequalities are not applicable.
Instead, here we will consider linear underestimates from the extended polymatroid associated with submodular $q$.

\begin{dfn}
  For a submodular set function $g:\binset^A \rightarrow \R$, the polyhedron
  $$
        EP_g \define \set{\bv\in \R^A}{\bv(S) \leq g(S),\,\,S\subseteq A}
  $$
  is called the extended polymatroid associated with $g$.
\end{dfn}
\vspace{1mm}
\begin{thm}
\citep{Edmonds1970submodular} For a submodular function $g$ with $g(\emptyset) = 0$, each vertex of the extended polymatroid $EP_g$ is given by
   $$
        v^{\bpi}_j = g(S^{\bpi}_j) - g(S^{\bpi}_{j-1})
   $$
   for a permutation $\bpi$ of $A$, where $S^{\bpi}_j$ denotes the set consisting of the first $j$ elements of the permutation $\bpi$.
\end{thm}

Let $\mathbf{V}_g$ be the set of all extreme points of $EP_g$ and
    consider the discrete epigraph of the set function $g$
    $$
        B_g \define \set{(\x, z) \in \binset^A\times \R}{g(x)\leq z}.
    $$
    It follows from polarity that any inequality
     $$\bv' \x \leq z, \,\,\bv \in EP_g$$ is valid for $B_g$ \citep[e.g.][]{AN:conicobj}. 
Moreover, inequalities $\bv' \x \leq z, \,\,\bv \in EP_g$ are sufficient to define  $\conv(B_g)$ \citep{Edmonds1970submodular}.
Therefore, the discrete lower level set of $g$ for a fixed value of $z$ can be stated as follows.     

\begin{prop}
\label{cor:extpolymat2}
  For $\bv \in \mathbf{V}_g$ and fixed $\gamma \in \R$, let $K_{\bv} \define \set{\x \in \binset^n}{\bv' \x \leq \gamma}$.  Then
  $$K_{g} \define \set{\x \in \binset^n}{g(\x) \leq \gamma} = \underset{\bv\inset \mathbf{V}_g}{\bigcap}\, K_{\bv}.$$
\end{prop}

 Proposition~\ref{cor:extpolymat2} allows us to linearize the submodular capacity constraint \eqref{eq:q} using inequalities from the vertices of the extended polymatroid $EP_q$. Although the number of such inequalities is factorial in $|A|$, we can use a separation algorithm to generate violated ones on the fly.
 Given a point $\bar \x$, we solve
 $$
 \max \{ \bar \x' \bv: \bv \in \mathbf{V}_q \}
 $$
exactly to find the strongest polymatroid knapsack inequality $\bv' \x \leq \gamma$ using the greedy algorithm of \citet{Edmonds1970submodular}.
Note that each $K_{\bv}$ is a $0-1$ knapsack set, for which
a large class of valid inequalities are known. To strengthen the formulation further, we generate cover inequalities for $K_{\bv}$ and
lift them using the superadditive lifting functions \cite{Atamturk2005}. As an additional step, we aggregate multiple polymatroid knapsack inequalities and
generate cover inequalities for the aggregated knapsack set and lift them in the same way. We present our computational analysis of the aforementioned approach in Section \ref{sec:Computations}.

\section{Separation}
\label{sec:Separation}

As \CQND \ contains exponentially many probabilistic capacity constraints, we utilize a cut generation algorithm to add only a small number of such constraints that are necessary. We let the initial relaxation of the problem to be the nominal version using only the mean capacities by ignoring the uncertainties in the capacity. This relaxation can be solved as a standard network design problem. Then we iteratively add violated probabilistic capacity constraints until no such violated inequality exists.

In order to find violated constraints for a given solution $\bar \x$ to a relaxation of the problem,
one needs to utilize a separation approach; that is,
to find a cut $C$ of the network
for which the corresponding probabilistic capacity constraint is violated:
\begin{align}
\label{eq:cut-sep}
 \bs{\mu}_C' \bar \x_C - \Omega\sqrt{\bar \x_C' {\bs \Sigma}_C\, \bar {\x}_C} < d.
\end{align}

We now formalize the separation problem for the probabilistic capacity constraints.
Let diag$(\bar \x)$ be the diagonalization of the vector $\bar \x$ and
${\bs \mu}_{\bar{\x}} = \text{diag}(\bar \x) \bs \mu$ and
$\boldsymbol{\Sigma_{\bar{\x}}} = \text{diag}(\bar \x) \bs \Sigma \text{diag}(\bar \x)$.
Further, let the binary variable $z_{ij}$ be 1 if arc $(ij) \in A$ is in the chosen cut and 0 otherwise;
and $w_i$ be 1 if node $i$ is on the source ($s$) side and 0 if it is on the sink ($t$) side.
Then, a cut $C$ achieving the smallest left-hand-side for inequality \eqref{eq:cut-sep} can be found by
solving the following nonlinear min-cut problem:

\begin{optprog}
  min & \objective{\Theta(\z) = {\bs \mu}_{\bar{\x}}'\mathbf{z}- \Omega\sqrt{\mathbf{z}'\boldsymbol{\Sigma_{\bar{\x}}}\mathbf{z}}} \nonumber\\
  s.t. & w_i -w_j &\leq& z_{ij},  & \,(ij) \in {{A}}\label{cons:1}\\
                    & w_i &\geq& z_{ij},  &\,(ij) \in {{A}} \label{cons:2}\\
  (\textbf{SEP}) \quad   & w_j &\leq& 1-z_{ij},  &\,(ij) \in {{A}} \label{cons:3}\\
                    & w_s &=& 1\label{cons:4}\\
                    & w_t &=& 0\label{cons:5}\\
                    & w_i,w_j &\geq& 0,  &\,(ij) \in {{A}}\label{cons:6}\\
                    & z_{ij} &\in& \{0,1\},   &\,(ij) \in {{A}}\label{cons:7}.
\end{optprog}
Note that $\Theta(\z)$ is a concave function as $\Omega \ge 0$ and $\boldsymbol{\Sigma_{\bar{\x}}}$ is positive semidefinite.
Constraints \eqref{cons:1}--\eqref{cons:3} ensure that only the arcs from the source ($s$) set to the sink ($t$) set are included in the chosen cut. Notice that for the usual (deterministic) min-cut problem with $\Omega$ = 0, as the objective is nonnegative, inequality \eqref{cons:1} is sufficient. However, one needs  constraints \eqref{cons:2} and \eqref{cons:3} to ensure the correctness of the formulation with $\Omega > 0$. Observe that (\textbf{SEP}) is \NP-hard, since for ${\bs \mu} = \bf{0}$ and diagonal $\bs{\Sigma}$ it reduces to the max-cut problem.

In order to have an effective method for solving \CQND, it is imperative to be able generate violated probabilistic capacity inequalities fast.
In the following subsections we discuss solving the separation problem (\textbf{SEP}) for the independent capacities and correlated capacities
separately as the correlations change the structure of the separation problem significantly.

\subsection{Independent arc capacities}
When the arc capacities are independent, the covariance matrix is diagonal; therefore, the separation problem reduces to
\begin{optprog*}
  min & \objective{\Theta(\z) = {\bs \mu}_{\bar{\x}}'\mathbf{z}
  - \Omega\sqrt{\sum_{a \in A}{\sigma_a^2}\bar{x}_a^2\,\,z_a }} \nonumber\\[-2mm]
  (\textbf{SEP-D}) \ \ \  \\[-2mm]
   & \quad \text{s.t.} & \eqref{cons:1} - \eqref{cons:7}.
\end{optprog*}

In this case the objective function $\Theta$ is convex, moreover, a non-decreasing supermodular function. Nevertheless, minimizing the set function $\Theta$ over binaries is \NP-hard \cite{Ahmed2011}. \citet{fishernemhausewolsey} show that any supermodular function minimization problem can be formulated using linear inequalities:
\begin{align}
    \min w & \ \ \text{s.t.} \nonumber \\
	w & \geq \Theta(S) - \sum_{i\in S} \rho_i(A\sm i)(1-z_i) + \sum_{i \in N\sm S} \rho_i(S)z_i,\,\, \forall S \subseteq A \label{ineq:submod1}\\
	w & \geq \Theta(S) - \sum_{i\in S} \rho_i(S\sm i)(1-z_i) + \sum_{i \in N\sm S} \rho_i(\emptyset)z_i,\,\, \forall S \subseteq A \label{ineq:submod2}
\end{align}
However, the linear mixed 0--1 formulation with inequalities \eqref{ineq:submod1}--\eqref{ineq:submod2} is
computationally ineffective as its convex relaxation is usually weak \citep{Ahmed2011}.

Alternatively, (\textbf{SEP-D}) can be formulated as a quadratically-constrained mixed 0--1 optimization problem by
bringing the nonlinear term in the objective into the constraints and squaring it:
\begin{optprog}
  min & \objective{\Theta(\z,t) = {\bs \mu}_{\bar{\x}}'\mathbf{z} - \Omega t} \nonumber \\
  (\textbf{SEP-QCQP})\quad s.t. & t^2 &\leq& \,{\sum_{a \in A}{\sigma_a^2}\bar{x}_a^2\,\,z_a} & \label{cons:t-sq}\\
  					& \eqref{cons:1} &-& \eqref{cons:7}. \nonumber
\end{optprog}

Note that \eqref{cons:t-sq} is a convex constraint with only single quadratic term and can be readily handled by commercial solvers. Table~\ref{tab:MIQP} shows a comparison of the solution of the two formulations with \solver{CPLEX} version 12.6 on a $2.93$GHz Pentium Linux workstation with $8$GB main memory. Each row shows the average number of branch-and-bound nodes till optimality and the time to the first feasible solution and the time to proving  optimality (in seconds) for five instances with varying network sizes and $\Omega$. Observe that (\textbf{SEP-QCQP}) is solved to optimality much faster than the linearization of (\textbf{SEP-D}) with \eqref{ineq:submod1}--\eqref{ineq:submod2}. For all most all instances it requires no branching at all. Furthermore, in all cases feasible solutions are found much earlier. Therefore, we use formulation (\textbf{SEP-QCQP}) for the independent case in the computational experiments presented in Section~\ref{sec:Computations}.

\begin{table}[h!]
	\centering
	\caption{Separation formulations: independent case.}
	\resizebox{0.8\textwidth}{!}{%
	\begin{tabular}{ccc|ccc|ccc}
    \hline \hline 
    \multirow{2}{*}{nodes} & \multirow{2}{*}{arcs} & \multirow{2}{*}{$\Omega$} & \multicolumn{3}{c|}{ SEP-D with ineqs. \eqref{ineq:submod1}--\eqref{ineq:submod2} } & \multicolumn{3}{c}{SEP-QCQP} \\ \cline{4-9}
          &       &       & nodes & feasible & optimal & nodes & feasible & optimal \\
    \hline 
    \multirow{3}{*}{50} & \multirow{3}{*}{1225} & 1 & 15 & 0.4 & 0.9 & 0  & 0.1  & 0.4  \\
          &       & 3   & 18  & 0.8 & 1.1 & 0   & 0.2  & 0.4 \\
          &       & 5   & 21  & 0.8 & 1.5   & 0   & 0.2  & 0.4   \\
    \hline 
    \multirow{3}{*}{70} & \multirow{3}{*}{2415} & 1  & 12  & 1.8 & 3.2 & 0  &   0.9 & 1.4  \\
          &       & 3     & 17 & 2.9  & 4.4 & 0  & 0.6  & 1.4   \\
          &       & 5     & 30 & 5.1 & 6.4 & 0  & 0.9  & 1.4   \\
    \hline 
    \multirow{3}{*}{100} & \multirow{3}{*}{4950} & 1 & 14  & 9.1 & 16.9    & 0   & 2.7 & 5.3  \\
          &       & 3      & 25  & 12.9 & 20.3 & 0  & 2.9 & 5.6  \\
          &       & 5    & 41  & 18.4 & 26.6 & 0  & 2.4  & 5.6  \\
          \hline 
          \multicolumn{3}{c|}{\textbf{avg}} & \textbf{21.4} & \textbf{5.8} & \textbf{9.0} & \textbf{0} & \textbf{1.2} & \textbf{2.4}\\
                   \multicolumn{3}{c|}{\textbf{stdev}} & \textbf{9.3} & \textbf{6.4} & \textbf{9.7} & \textbf{0} & \textbf{1.1} & \textbf{2.3}\\
    \hline \hline 
    \end{tabular}}%
    \label{tab:MIQP}	
\end{table}

\subsection{Correlated arc capacities}

For the general case with correlated arc capacities, as the covariance matrix is not diagonal, the objective of the separation problem is no longer supermodular nor can it
be formulated as (\textbf{SEP-QCQP}) due to the bilinear terms. To overcome the difficulty, we present two approaches. The first one is the usual McCormick linearization of the bilinear terms, whereas the second one is a two-step approach that allows one to convert the separation problem into one with a quadratic constraint on binary variables.

\textit{McCormick linearization}

 The classical approach of \citet{McCormick76} to solving non-convex quadratic optimization problems is to linearize the bilinear terms using additional variables and constraints. In particular, a bilinear term $z = x_1x_2$ is re-formulated with the \emph{McCormick inequalities}:
\begin{align}
\label{ineq:mccormick}
  z \geq \max\,(x_1+x_2-1, 0),\hspace{0.5cm} z \leq \min\,(x_1,x_2).
\end{align}

Now consider the objective function of the separation problem (\textbf{SEP}):
\begin{align} \label{eq:theta-orig}
   \Theta(\z) = {\bs \mu}_{\bar{\x}}'\mathbf{z}- \Omega\sqrt{\mathbf{z}'\boldsymbol{\Sigma_{\bar{\x}}}\mathbf{z}} = {\bs \mu}_{\bar{\x}}'\mathbf{z}- \Omega\sqrt{\sum_{i}(\sigma_i \bar x_i)^2z_i^2 + 2\sum_i\sum_{j>i} \sigma_{ij} \bar{x}_{i} \bar{x}_{j}  z_iz_j}.
\end{align}
Introducing a new variable $y_{ij}$ for each distinct pair $\{i,j\}$, we formulate the separation problem as
\begin{optprog}
min & \objective{\Theta(\z, \y, t ) = {\bs \mu}_{\bar{\x}}'\mathbf{z}- \Omega t} \nonumber \\
s.t. &  \objective{ t^2 \le \sum_{i}(\sigma_i \bar x_i)^2z_i + 2\sum_i\sum_{j>i} \sigma_{ij} \bar{x}_i \bar x_j y_{ij} } \label{cons:mc0} \\
(\textbf{SEP-MC}) \ \ \ & z_i + z_j &\leq & y_{ij} + 1, & \,\,i, j \in {A} \label{cons:mc1}\\
  & y_{ij} &\leq & z_i,  & \,i, j \in {A}\label{cons:mc2}\\
  & y_{ij} &\leq &  z_j,  & \,i, j \in {A}\label{cons:mc3}\\
  & z_i &\in &  \{0,1\},  & \,i \in {A}\label{cons:mc4}\\
  & \eqref{cons:1} &-& \eqref{cons:7}. \nonumber
\end{optprog}

Observe that while \eqref{eq:theta-orig} is a concave function in $z$, \eqref{cons:mc0} is a convex constraint with a single quadratic term, and, thus,
(\textbf{SEP-MC}) can be readily solved by conic quadratic MIP solvers. Also note that
constraints \eqref{cons:mc1}--\eqref{cons:mc4} implicitly restrict $y_{ij}$ to be binary. However,
 a challenge with this McCormick linearization is that it often yields a weak relaxation of the original formulation. Moreover, the large number (quadratic in $|A|$) of auxiliary variables and constraints introduced for the linearization makes it difficult to scale up the dimension of the problem.

\textit{A quadratic reformulation}

In this section we give an alternative reformulation of the separation problem that does not require a large number of auxiliary variables. In particular, we reformulate the separation problem with correlations using a quadratic set function as in Section \ref{sec:corrMaster}.

Note that for a given $\bar \x$, the cut corresponding to a feasible solution $\z$ to ({\textbf{SEP}) is satisfied if and only if
\begin{align}
   {\bs \mu}_{\bar{\x}}'\mathbf{z}- \Omega\sqrt{\mathbf{z}'\boldsymbol{\Sigma_{\bar{\x}}}\mathbf{z}} \ge d.
\end{align}
Assuming ${\bs \mu}_{\bar{\x}}'\mathbf{z} \ge d$, this is equivalent to
\begin{align}
   \big ({\bs \mu}_{\bar{\x}}'\mathbf{z}- d \big)^2 \ge \Omega^2 \mathbf{z}'\boldsymbol{\Sigma_{\bar{\x}}}\mathbf{z}.
\end{align}

Observe that it is easy to satisfy the condition  ${\bs \mu}_{\bar{\x}}'\mathbf{z} \ge d$ as it corresponds to solving a standard min-cut problem. Therefore, in the first stage, we ensure that $\bar \x$ satisfies that standard (linear) cut capacity constraints ${\bs \mu}' \x \ge d$ and then, in the second stage, we look for a violated probabilistic capacity constraint by solving
\begin{optprog}
      minimize & \objective{{\bs \mu}_{\bar{\x}}'\mathbf{z}}\nonumber\\
      (\textbf{SEP-BQC})\quad \quad s.t. & ({\bs \mu}_{\bar{\x}}'\mathbf{z} - d)^2 &\leq& \Omega^2 \cdot \mathbf{z}'\boldsymbol{\Sigma_{\bar{\x}}}\mathbf{z} - \epsilon \label{cons:qcp}\\
                        & \eqref{cons:1} - \eqref{cons:7} & &\nonumber
\end{optprog}
for a sufficiently small $\epsilon > 0$. Note that any feasible solution to (\textbf{SEP-BQC}) corresponds to a violated probabilistic cut constraint.
As \eqref{cons:qcp} is a quadratic constraint on binary variables, it can be readily solved by commercial MIP solvers. Observe that (\textbf{SEP-BQC}) is a much smaller formulation than (\textbf{SEP-MC}).

In the following, we compare the computational performance for three solution approaches:
\begin{enumerate}
	\item (\textbf{SEP-MC}) with full McCormick linearization
	\item (\textbf{SEP-BQC}) with partial McCormick linearization of only the non-convex right-hand-side of constraint \eqref{cons:qcp}.
	\item (\textbf{SEP-BQC}) with no linearization.
\end{enumerate}

Each row of Table~\ref{tab:QCP} shows the average for five randomly generated instances with varying network size and $\Omega$. We compare the number of nodes explored in the search tree and the time (in seconds) to the first feasible solution and optimality.

\begin{table}[h!]
  \centering
  \caption{Separation formulations: correlated case.}%
  \resizebox{\textwidth}{!}{%
    \begin{tabular}{ccc|ccc|ccc|ccc}
    \hline \hline 
    \multirow{2}{*}{nodes} & \multirow{2}{*}{arcs} & \multirow{2}{*}{$\Omega$} & \multicolumn{3}{c|}{Full McCormick} & \multicolumn{3}{c|}{Partial McCormick} & \multicolumn{3}{c}{{SEP-BQC}} \\ \cline{4-12}
          &       &       & nodes & feasible & optimal & nodes & feasible & optimal & nodes & feasible & optimal \\
    \hline 
          &       & 1     & 2 & 1.4 & 3.6   & 5 & 1.1 & 5 & 		0 & 0  & 0.2 \\[1mm]
    20    & 105   & 3     & 0 & 1.8 & 4.3   & 14 & 3.1 & 10.9 & 		0 & 0.1  & 0.2 \\[1mm]
          &       & 5     & 2 & 1.7 & 4.2     & 31 & 4.4 & 12.7 &		1 & 0.1  & 0.2 \\
    \hline 
          &       & 1     & 11 & 42.1  & 54.6   & 32 & 22.9 & 61.2 &		5 & 0.2   & 0.6 \\[1mm]
    30    & 231   & 3     & 17 & 51.7 & 95.2    & 366 & 54.5 & 88.6 & 		4 & 0.6   & 0.6 \\[1mm]
          &       & 5     & 20 & 61.6 & 122.1   & 1093 & 36.6 & 190 & 		3 & 1.0   & 5.4 \\
    \hline 
          &       & 1     & 36 & 65.2  & 676.7    & 39 & 188.9 & 283 & 		1 & 4.3     & 5.3 \\[1mm]
    40    & 432   & 3     & 51 & 66.7  & 816.3    & 150 & 446.3 & 498 & 		2 & 2.0     & 9.8 \\[1mm]
          &       & 5     & 25 & 72.8  & 736.5    & 291 & 447.1 & 628 &  	5  & 		3.9	& 11.2 \\
     \hline 
          \multicolumn{3}{c|}{\textbf{avg}} & \textbf{18.2} & \textbf{40.6} & \textbf{279.3} & \textbf{224.6} & \textbf{133.9} & \textbf{197.6} & \textbf{2.3} & \textbf{1.0}  & \textbf{3.7} \\
                    \multicolumn{3}{c|}{\textbf{stdev}} & \textbf{17.1} & \textbf{30.5} & \textbf{352.1} & \textbf{351} & \textbf{186.5} & \textbf{229.1} & \textbf{2.0} & \textbf{1.5} & \textbf{4.4}\\
  \hline \hline 
    \end{tabular}}%
  \label{tab:QCP}%
\end{table}%

It is evident from the computations displayed in Table \ref{tab:QCP} that (\textbf{SEP-BQC}) reformulation leads to a much better computational performance compared to the partial or full McCormick linearizations. As the separation problem is called many times during the solution of the probabilistic network design problem, shorter separation times can make a significant impact on the overall solution times. Therefore, we use formulation (\textbf{SEP-BQC}) for the separation for the correlated case.

\section{Computations}
\label{sec:Computations}
In this section we present our computational experiments for solving the network design problem with probabilistic capacities.
We compare the impact of the strengthening methods discussed in Section~\ref{sec:strong} in the computational performance.
We utilize the solution approaches for the separation problems discussed in Section~\ref{sec:Separation}.
For the computational experiments we use the MIP solver of \solver{CPLEX} version 12.6 that solves conic quadratic relaxations at the nodes of the branch-and-bound tree. \solver{CPLEX} heuristics are turned off, and a single thread is used. The MIP search strategy is set to the traditional branch-and-bound method as it is not possible to add users cuts in \solver{CPLEX} with the dynamic search strategy. Since the solution approaches are significantly different for the independent and correlated cases, we address them separately. For the case of independent capacities the time limit and the memory limit are set to $1800$ secs. and $500$ MB, respectively, whereas for the harder correlated case, the limits are $3600$ secs. and $1$ GB. The experiments are performed on a $2.93$GHz Pentium Linux workstation with $8$GB main memory.

\subsection{Data generation}
\label{sec:data}
We report the results of the computational experiments for varying number of nodes ($n$) and arcs ($a$), values of $\Omega$, and demand. For each combination, five random networks are generated with mean arc capacities drawn from uniform $[0, 100]$ and $\sigma_i$ from uniform $\left[0, {\mu_i}/{\Omega}\right]$. The demand $d = \beta\cdot \phi$, where $\phi$ is the maximum flow possible through the network for the given value of $\Omega$.
So that the networks are not completely dense, we set the probability of having an arc between two nodes as ${1}/{\sqrt{n}}$, while ensuring the connectedness of the network. In the case of correlated capacities, for each problem instance, the covariance matrix is generated using MATLAB moler matrix library which generates random symmetric positive definite matrices. The eigenvalues of each generated covariance matrix are then translated by half the spectral radius to reduce the condition number. The mean arc capacities in the correlated case are generated from uniform $[\Omega\,\sigma_i,\,2\,\Omega\,\sigma_i]$. In Section \ref{sec:relax-assumption}, we drop the assumption \ref{asmptn:coeffofvar} on the coefficient of variation and the mean capacities are generated from uniform $[0,\,2\,\Omega\,\sigma_i]$. The data set is available for download at \texttt{http://ieor.berkeley.edu/$\sim$atamturk/data/conic.network.design} .

\subsection{Independent arc capacities}
In Table \ref{table:networkFinal} we compare the root relaxation gap (\texttt{rgap}), the numbers of cuts generated (\texttt{cuts}), the number of nodes explored \texttt{(nodes)}, the CPU time in seconds (\texttt{time}) and the number of instances solved to optimality (\texttt{\#}) with four cut generation options. The \texttt{rgap} is computed as $(z_r - z_o)/{z_o}$, where $z_r$ denotes the objective value at the root node and $z_o$ denotes the optimal objective value. If none of the algorithms solve a given instance to optimality within the computation limits, then $z_o$ is the objective value of the best found solution across all algorithms. Furthermore, if an algorithm is unable to solve all of the five instances, we provide the end gap (\texttt{egap}), which is computed as $(z_e - z_o)/{z_o}$, where $z_e$ is the best objective value found with the algorithm. Each entry in the table presents the average for five instances, except for the \texttt{\#[egap]} column, where the \texttt{egap} is the average for the unsolved instances.

 The columns under the heading `conic cuts' show the performance with the \solver{CPLEX} barrier algorithm when the conic quadratic constraints \eqref{eq:equivalence} added to the formulation through the separation routine. We compare it with two linearizations: the first one is the CPLEX outer approximation, where CPLEX automatically generates gradients cuts and solves linear programs iteratively instead of conic quadratic programs with a barrier algorithm; whereas the second one is our implementation of the gradient cuts. Both linearizations perform better than using the barrier algorithm. CPLEX generates additional cutting planes to these linear formulations to improve the relaxations. Note that our implementation of the gradient cuts is more aggressive and generates significantly more cuts and reduces the root gaps and the number of nodes substantially; however, it is not as efficient as the native CPLEX implementation and takes longer time. Therefore, for the rest of the computations, we use the CPLEX outer approximation. The last set of columns under the heading `extended packs' show the performance with the addition of extended pack inequalities \citep{ab:packs}. The use of extended pack cuts strengthens the formulations substantially by reducing the average root gap to 12\% and leads to much shorter solution times.

\begin{sidewaystable}[p]
\vskip 10cm
  \caption{\small Network design: independent case.}
    \resizebox{\textwidth}{!}{%
    \renewcommand{\arraystretch}{1.2}
    \begin{tabular}{cccc|ccccc|ccccc|ccccc|ccccc}
    \hline \hline 
          &       &       &       & \multicolumn{5}{c}{Conic cuts}        & \multicolumn{5}{c}{CPLEX outer approximation}  & \multicolumn{5}{c}{User outer approximation} & \multicolumn{5}{c}{Extended packs}\\
   \hline 
    nodes & arcs  & $\beta$  & $\Omega$ & rgap  & cuts  & nodes & time  & \# [egap]  & rgap  & cuts  & nodes & time  & \# [egap]    & rgap  & cuts  & nodes & time  & \# [egap]    & rgap  & cuts  & nodes & time  & \# [egap] \\
 \hline 
    \multirow{9}{*}{10} & \multirow{9}{*}{18} & \multirow{3}{*}{0.3} & 1     & {66.2} & 3     & {123} & {1} & {5} & {66.2} & 3     & {82} & {1} & {5} & {60.42} & 87    & {24} & {76} & {5} & {0} & 7     & {0} & {4} & {5} \\
          &       &       & 3     & {68.5} & 7     & {394} & {1} & {5} & {68.5} & 7     & {226} & {3} & {5} & {64.9} & 89    & {21} & {71} & {5} & {12.0} & 11    & {1} & {6} & {5} \\
          &       &       & 5     & {60.4} & 4     & {349} & {1} & {5} & {60.4} & 4     & {163} & {2} & {5} & {53.4} & 106   & {19} & {80} & {5} & {1.3} & 9     & {0} & {3} & {5} \\
          &       & \multirow{3}{*}{0.5} & 1     & {54.4} & 5     & {264} & {1} & {5} & {49.0} & 5     & {158} & {0} & {5} & {46.4} & 77    & {20} & {62} & {5} & {20.5} & 9     & {1} & {5} & {5} \\
          &       &       & 3     & {59.7} & 6     & {408} & {3} & {5} & {59.7} & 6     & {161} & {4} & {5} & {51.5} & 66    & {19} & {60} & {5} & {0.2} & 11    & {0} & {6} & {5} \\
          &       &       & 5     & {44} & 8     & {1817} & {8} & {5} & {44} & 8     & {1187} & {2} & {5} & {39.0} & 125   & {31} & {94} & {5} & {2.3} & 10    & {1} & {7} & {5} \\
          &       & \multirow{3}{*}{0.7} & 1     & {42.2} & 6     & {386} & {1} & {5} & {40.9} & 6     & {278} & {2} & {5} & {38.9} & 82    & {18} & {68} & {5} & {3.0} & 11    & {1} & {4} & {5} \\
          &       &       & 3     & {39.0} & 3     & {394} & {1} & {5} & {41.4} & 3     & {206} & {0} & {5} & {32.6} & 134   & {27} & {101} & {5} & {9.6} & 9     & {3} & {3} & {5} \\
          &       &       & 5     & {56.9} & 11    & {1251} & {6} & {5} & {53.1} & 11    & {876} & {5} & {5} & {52.3} & 113   & {23} & {88} & {5} & {0.8} & 12    & {1} & {4} & {5} \\
  \hline 
    \multirow{9}{*}{20} & \multirow{9}{*}{54} & \multirow{3}{*}{0.3} & 1     & {70.8} & 21    & {28146} & {628} & {4 [26.1]} & {68.7} & 34    & {16424} & {16} & {5} & {60.9} & 836   & {81} & {398} & {5} & {8.0} & 102   & {6} & {8} & {5} \\
          &       &       & 3     & {68.5} & 27    & {19620} & {591} & {5} & {68.5} & 29    & {3570} & {5} & {5} & {61.9} & 562   & {61} & {285} & {5} & {1.5} & 46    & {2} & {12} & {5} \\
          &       &       & 5     & {67.3} & 23    & {19950} & {773} & {3 [25.3]} & {67.2} & 28    & {3236} & {6} & {5} & {51.6} & 469   & {42} & {237} & {5} & {7.2} & 60    & {3} & {12} & {5} \\
          &       & \multirow{3}{*}{0.5} & 1     & {56.5} & 16    & {17876} & {639} & {4 [19.3]} & {52.0} & 32    & {22242} & {26} & {5} & {44.3} & 815   & {97} & {368} & {5} & {9.2} & 164   & {13} & {36} & {5} \\
          &       &       & 3     & {61.1} & 26    & {29287} & {887} & {3 [23.2]} & {56.9} & 53    & {51336} & {54} & {5} & {53.1} & 763   & {93} & {315} & {5} & {11.1} & 141   & {17} & {20} & {5} \\
          &       &       & 5     & {59.5} & 22    & {21371} & {760} & {4 [18.6]} & {56.0} & 40    & {25241} & {25} & {5} & {50.5} & 1124  & {120} & {490} & {5} & {13.3} & 298   & {30} & {28} & {5} \\
          &       & \multirow{3}{*}{0.7} & 1     & {51.8} & 21    & {38152} & {769} & {3 [22.1]} & {47.5} & 43    & {49404} & {49} & {5} & {41.7} & 801   & {122} & {416} & {5} & {8.7} & 150   & {13} & {27} & {5} \\
          &       &       & 3     & {51.4} & 24    & {33386} & {761} & {4 [20.2]} & {46.7} & 36    & {32124} & {22} & {5} & {45.1} & 673   & {102} & {326} & {5} & {7.6} & 125   & {12} & {26} & {5} \\
          &       &       & 5     & {51.3} & 22    & {53640} & {762} & {3 [20.8]} & {44.6} & 46    & {47598} & {40} & {5} & {38.2} & 579   & {93} & {300} & {5} & {9.3} & 80    & {6} & {10} & {5} \\
   \hline 
    \multirow{9}{*}{40} & \multirow{9}{*}{153} & \multirow{3}{*}{0.3} & 1     & {64.7} & 18    & {46877} & {1450} & {2 [25.5]} & {58.63} & 122   & {297893} & {624} & {4 [19.7]} & {42.0} & 2558  & {127} & {1729} & {1 [15.2]} & {14.5} & 819   & {46} & {66} & {5} \\
          &       &       & 3     & {66.6} & 17    & {46226} & {1440} & {2 [27.2]} & {60.3} & 162  & {457997} & {936} & {3 [20.3]} & {48.3} & 1938  & {119} & {1323} & {2 [15.9]} & {14.8} & 1230  & {69} & {97} & {5} \\
          &       &       & 5     & {68.3} & 16    & {49432} & {1473} & {1 [31.2]} & {59.7} & 138   & {586878} & {1118} & {2 [20.8} & {52.4} & 2391  & {124} & {1629} & {1 [16.4]} & {21.7} & 737   & {47} & {53} & {5} \\
          &       & \multirow{3}{*}{0.5} & 1     & {58.7} & 18    & {87071} & {1800} & {1 [27.6]} & {48.1} & 142   & {898215} & {1222} & {4 [18.7]} & {37.8} & 2526  & {138} & {1602} & {2 [10.7]} & {20.5} & 1458  & {89} & {131} & {5} \\
          &       &       & 3     & {62.3} & 17    & {82869} & {1800} & {1 [29.2]} & {54.5} & 179   & {1008149} & {1603} & {3 [19.1]} & {44.7} & 2608  & {143} & {1740} & {1 [15.1]} & {27.9} & 1485  & {169} & {177} & {5} \\
          &       &       & 5     & {61.2} & 16    & {59570} & {1800} & {0 [29.8]} & {51.3} & 178   & {906390} & {1581} & {2 [20.9]} & {41.9} & 2845  & {189} & {1946} & {1 [14.8]} & {24.4} & 2155  & {247} & {424} & {4 [4.3]} \\
          &       & \multirow{3}{*}{0.7} & 1     & {52.6} & 16    & {64761} & {1800} & {2 [25.6]} & {42.2} & 148   & {806376} & {1184} & {4 [17.1]} & {36.2} & 2378  & {129} & {1639} & {1 [11.4]} & {27.1} & 2357  & {227} & {428} & {4 [5.2]} \\
          &       &       & 3     & {57.8} & 16    & {69720} & {1800} & {1 [27.2]} & {42.7} & 121   & {998075} & {1125} & {3 [18.0]} & {35.6} & 2169  & {121} & {1398} & {2 [10.5]} & {26.4} & 2304  & {204} & {436} & {4 [4.5]} \\
          &       &       & 5     & {55.7} & 15    & {83062} & {1800} & {0 [30.3]} & {45.1} & 152   & {772617} & {1023} & {3 [18.6]} & {35.31} & 2281  & {175} & {1515} & {3 [10.5]} & {24.9} & 2201  & {254} & {432} & {4 [3.8]} \\
   \hline 
          \multicolumn{4}{c|}{\textbf{avg}} & \textbf{58.4} & \textbf{15} & \textbf{31719} & \textbf{806} &  & \textbf{53.8} & \textbf{64} & \textbf{258782} & \textbf{396} & & \textbf{46.7} & \textbf{1081} & \textbf{84} & \textbf{670} & & \textbf{12.1} & \textbf{593} & \textbf{54} & \textbf{91} & \\
          \multicolumn{4}{c|}{\textbf{stdev}} & \textbf{8.3} & \textbf{7} & \textbf{29239} & \textbf{710} &  & \textbf{9.1} & \textbf{64} & \textbf{378896} & \textbf{574} &  & \textbf{8.9} & \textbf{1011} & \textbf{53} & \textbf{691} &  & \textbf{8.9} & \textbf{835} & \textbf{85} & \textbf{150} & \\
  \hline \hline 
    \end{tabular}}%
  \label{table:networkFinal}%
\end{sidewaystable}%

\subsection{Correlated arc capacities}

In Table~\ref{tab:networkcorr} we compare the performance with conic cuts, CPLEX outer aproximation as well as with adding the
extended polymatroid cuts, and the aggregated lifted cover cuts to the formulations. As in the independent case, solving iterative linear programs via CPLEX outer approximation is substantially faster compared to solving conic the quadratic programs with a barrier algorithm in the search tree.  The user cuts are generated throughout the search tree using the respective separation algorithms on top of the CPLEX outer approximation. Table~ \ref{tab:networkcorr} clearly shows the positive impact of strengthening the formulations with the
extended polymatroid cuts and the aggregated lifted cover cuts. All instances are solved to optimality with the addition of user cuts and the average solution time across all instances reduces from 1587 seconds to merely 4 seconds.
Note that the separation problems for network instances in Table~\ref{tab:networkcorr} are solved much faster than the ones in Table~\ref{tab:QCP} because (a) instances in Table~\ref{tab:networkcorr} are smaller and (b) the separation problems are solved only on subsets of the arcs with positive flow.

\begin{sidewaystable}[p]
\vskip 10cm
  \caption{\small Network design: correlated case.}
    \resizebox{\textwidth}{!}{%
    \renewcommand{\arraystretch}{1.3}
    \begin{tabular}{cccc|ccccc|ccccc|ccccc|cccccc}
    \hline \hline
          &       &       &       & \multicolumn{5}{c|}{Conic cuts}             & \multicolumn{5}{c|}{CPLEX outer approximation}    & \multicolumn{5}{c|}{Extended polymatroid cuts}      & \multicolumn{6}{c}{Extended polymatroid \& aggregated lifted cover cuts} \\
    \hline \midrule
    nodes & arcs  & $\beta$  & $\Omega$ & rgap  & cuts  & nodes & time  & \# [egap]    & rgap  & cuts  & nodes & time & \# [egap]    & rgap  & cuts  & nodes & time & \# [egap]    & rgap  & cuts & covers & nodes & time  & \# [egap] \\
    \hline 
        \multirow{9}{*}{10} & \multirow{9}{*}{21} & \multirow{3}{*}{0.3} & 1     & {84.1} & 6     & {474} & {2.2} & {5 } & {84.1} & 6     & {232} & {0.1} & {5 } & {80.4} & 18    & {67} & {0.1} & {5} & {64.2} & 16    & 9     & {33} & {0.1} & {5} \\
          &       &       & 3     & {79.1} & 6     & {675} & {3.4} & {5 } & {79.1} & 5     & {280} & {0.1} & {5 } & {77.4} & 16    & {71} & {0.1} & {5} & {61.2} & 14    & 10    & {27} & {0.1} & {5} \\
          &       &       & 5     & {82.7} & 4     & {616} & {1.7} & {5 } & {82.7} & 4     & {202} & {0.1} & {5 } & {81.4} & 33    & {89} & {0.1} & {5} & {62.6} & 20    & 9     & {37} & {0.1} & {5} \\
       &       & \multirow{3}{*}{0.5} & 1     & {72.8} & 7     & {468} & {2.6} & {5 } & {72.8} & 7     & {314} & {0.1} & {5 } & {66.1} & 43    & {104} & {0.1} & {5} & {50.2} & 39    & 17    & {58} & {0.1} & {5} \\
          &       &       & 3     & {79.1} & 7     & {1314} & {5.5} & {5 } & {79.1} & 7     & {430} & {0.1} & {5 } & {74.2} & 30    & {119} & {0.3} & {5} & {59} & 30    & 16    & {62} & {0.1} & {5} \\
          &       &       & 5     & {73.9} & 8     & {901} & {10.5} & {5 } & {73.9} & 8     & {427} & {0.2} & {5 } & {70.2} & 54    & {174} & {0.2} & {5} & {48} & 39    & 21    & {70} & {0.2} & {5} \\
       &       & \multirow{3}{*}{0.7} & 1     & {68.3} & 7     & {1523} & {7.6} & {5 } & {68.3} & 7     & {721} & {0.2} & {5 } & {66.8} & 44    & {190} & {0.22} & {5} & {45.3} & 56    & 25    & {84} & {0.2} & {5} \\
          &       &       & 3     & {80} & 11    & {2670} & {20.4} & {5 } & {80} & 11    & {855} & {0.3} & {5 } & {73.5} & 53    & {204} & {0.4} & {5} & {52.6} & 52    & 26    & {98} & {0.2} & {5} \\
          &       &       & 5     & {76.1} & 12    & {5614} & {97.8} & {5 } & {76.1} & 12    & {2125} & {0.8} & {5 } & {70.4} & 98    & {359} & {0.3} & {5} & {49.3} & 101   & 42    & {141} & {0.2} & {5} \\
    \hline
    \multirow{9}{*}{20} & \multirow{9}{*}{54} & \multirow{3}{*}{0.3} & 1     & {73.1} & 29    & {25429} & {1327.9} & {4 [24]} & {72.2} & 11    & {5671} & {3.5} & {5 } & {72} & 279   & {1834} & {1.3} & {5} & {55} & 208   & 141   & {248} & {0.8} & {5} \\
          &       &       & 3     & {73.8} & 16    & {7690} & {238.7} & {5 } & {73.8} & 16    & {3837} & {2.2} & {5 } & {72} & 164   & {792} & {0.7} & {5} & {55.4} & 154   & 121   & {256} & {0.6} & {5} \\
          &       &       & 5     & {72.4} & 10    & {4283} & {81.5} & {5 } & {72.4} & 10    & {2784} & {1.2} & {5 } & {70.7} & 118   & {395} & {0.5} & {5} & {48.3} & 90    & 62    & {155} & {0.4} & {5} \\
       &       & \multirow{3}{*}{0.5} & 1     & {67.3} & 51    & {59031} & {2176.4} & {2 [9.9]} & {69.5} & 29    & {41528} & {33.7} & {5 } & {65.3} & 423   & {4758} & {3.2} & {5} & {46} & 466   & 200   & {1132} & {2.5} & {5} \\
          &       &       & 3     & {70.4} & 41    & {36523} & {1712.4} & {4 [4.1]} & {70.5} & 31    & {14067} & {8.4} & {5 } & {66.9} & 418   & {1793} & {1.7} & {5} & {49.4} & 344   & 167   & {531} & {1.5} & {5} \\
          &       &       & 5     & {68.6} & 84    & {34812} & {1583.5} & {3 [6.6]} & {70.4} & 50    & {129362} & {727.3} & {4 [11.2]} & {68.5} & 258   & {1074} & {1} & {5} & {50.6} & 190   & 120   & {310} & {0.8} & {5} \\
       &       & \multirow{3}{*}{0.7} & 1     & {60.7} & 25    & {38103} & {1254.5} & {4 [13.3]} & {62.2} & 33    & {46481} & {32.5} & {5 } & {62.1} & 719   & {8237} & {10.8} & {5} & {43.4} & 397   & 197   & {1270} & {2.6} & {5} \\
          &       &       & 3     & {67.3} & 66    & {102159} & {3600} & {0 [9.7]} & {68.4} & 44    & {28176} & {19.6} & {5 } & {65.5} & 352   & {2192} & {2.1} & {5} & {46.3} & 238   & 134   & {647} & {1.2} & {5} \\
          &       &       & 5     & {67} & 48    & {47892} & {1698.7} & {4 [14.9]} & {67.7} & 38    & {24963} & {24.3} & {5 } & {65.7} & 346   & {2658} & {1.9} & {5} & {47} & 393   & 224   & {828} & {2.4} & {5} \\
    \hline
    \multirow{9}{*}{30} & \multirow{9}{*}{101} & \multirow{3}{*}{0.3} & 1     & {68.1} & 46    & {32337} & {2854.5} & {2 [18.4]} & {71.3} & 68    & {56284} & {171.4} & {5 } & {70.6} & 231   & {2304} & {1.7} & {5} & {51} & 387   & 222   & {672} & {2.9} & {5} \\
          &       &       & 3     & {66.6} & 54    & {28538} & {2586.2} & {2 [23.1]} & {72.2} & 108   & {137584} & {844.8} & {4 } & {67.9} & 507   & {5620} & {5.2} & {5} & {50} & 397   & 246   & {814} & {3.2} & {5} \\
          &       &       & 5     & {67.6} & 100   & {33951} & {2735.2} & {2 [21.2]} & {71.6} & 87    & {77529} & {385.8} & {5 } & {69.3} & 305   & {2582} & {2.2} & {5} & {49.4} & 191   & 139   & {460} & {1.5} & {5} \\
   &       & \multirow{3}{*}{0.5} & 1     & {57.6} & 90    & {109833} & {3600} & {0 [18.8]} & {65.7} & 106   & {541575} & {2219.5} & {3 [4.0]} & {64.1} & 1244  & {24314} & {24.2} & {5} & {42} & 937   & 566   & {2596} & {9.5} & {5} \\
          &       &       & 3     & {59.6} & 21    & {65734} & {2850.7} & {2 [16.7]} & {64.3} & 83    & {330311} & {1612.2} & {4 [8.1]} & {60.9} & 1297  & {14424} & {15.9} & {5} & {42.8} & 1155  & 721   & {2999} & {14.7} & {5} \\
          &       &       & 5     & {60.8} & 40    & {133861} & {3600} & {0 [12.3]} & {65.2} & 98    & {262583} & {1123.6} & {5 } & {61.5} & 963   & {5186} & {5.5} & {5} & {39.8} & 866   & 479   & {1242} & {6.3} & {5} \\
         &       & \multirow{3}{*}{0.7} & 1     & {44.4} & 67    & {83056} & {3600} & {0 [30.7]} & {59.3} & 108   & {644090} & {2892.9} & {2 [9.2]} & {61.5} & 1790  & {51076} & {63.1} & {5} & {36.6} & 1686  & 962   & {4759} & {28.8} & {5} \\
          &       &       & 3     & {56.9} & 63    & {90960} & {3600} & {0 [10.8]} & {62.7} & 63    & {540993} & {1811.8} & {3 [22.4]} & {62.2} & 1977  & {49073} & {81.9} & {5} & {42} & 2168  & 1118  & {3411} & {27} & {5} \\
          &       &       & 5     & {51} & 83    & {76914} & {3600} & {0 [13.5]} & {56.6} & 91    & {314039} & {984.8} & {4 [6.1]} & {57} & 1175  & {13124} & {15.7} & {5} & {41.2} & 1379  & 630   & {2022} & {11.3} & {5} \\
    \hline
    \multicolumn{4}{c|}{\textbf{avg}} & \textbf{68.5} & \textbf{37} & \textbf{37976} & \textbf{1587} &  & \textbf{70.8} & \textbf{42} & \textbf{118794} & \textbf{478} &  & \textbf{68.3} & \textbf{480} & \textbf{7141} & \textbf{9} &  & \textbf{49.2} & \textbf{445} & \textbf{245} & \textbf{925} & \textbf{4} & \\
    \multicolumn{4}{c|}{\textbf{stdev}} & \textbf{9.5} & \textbf{30} & \textbf{39441} & \textbf{1476} & & \textbf{6.8} & \textbf{38} & \textbf{190837} & \textbf{801} & & \textbf{5.9} & \textbf{560} & \textbf{13578} & \textbf{19} &  & \textbf{5.8} & \textbf{563} & \textbf{302} & \textbf{1216} & \textbf{8} & \\
  \hline \hline 
    \end{tabular}}%
  \label{tab:networkcorr}%
\end{sidewaystable}%

\subsection{Relaxing the assumption on the coefficient of variation}
\label{sec:relax-assumption}

Throughout the paper, we assumed an upper bound on the coefficient of variation \ref{asmptn:coeffofvar}. While assumption~\ref{asmptn:coeffofvar} is a sufficient condition leading to the submodularity of the set functions used in the analysis, we show in this section, that it is possible to drop this assumption and still utilize the results where applicable.

Recall from Section~\ref{sec:corrMaster} that if \ref{asmptn:coeffofvar} holds, the coefficients $\beta_{ij}$ of the function
\begin{align*}
	q(\x) = \sum_i \alpha_ix_i +  \sum_i\sum_{j>i} 2 \beta_{ij}x_ix_j + d^2,
\end{align*}
are non-positive, and consequently, $q$ is submodular. We can drop this assumption by using the McCormick linearization for pairs $(i,j)$ where $\beta_{ij} > 0$. Introducing $z_{ij} = x_ix_j$ for such pairs only, consider the function
$$
	\tilde q(\x,\mathbf{z}) = \sum_i \alpha_ix_i + \sum_{ ij:j>i,\beta_{ij} < 0} 2\beta_{ij}x_ix_j +  \sum_{ij:j>i, \beta_{ij} > 0} 2 \beta_{ij}z_{ij} + d^2.
$$
As $\tilde q$ is submodular, we can employ the results using $\tilde q$ instead of $q$.

In Table \ref{tab:networkgeneral} we present the computational results for the general probabilistic network design instances without the  assumption \ref{asmptn:coeffofvar}.
In this data set, 29\% of the coefficients $\beta_{ij}$ are positive.
The results are largely consistent with Table~\ref{tab:networkcorr} with slight degradation in the performance. The cutting planes added reduce the average solution time from 2334 seconds to merely 10 seconds and confirms their effectiveness for the general case as well.

\begin{sidewaystable}[htbp]
\vskip 10cm
  \caption{\small Network design: the general case.}
    \resizebox{\textwidth}{!}{%
    \renewcommand{\arraystretch}{1.3}
    \begin{tabular}{cccc|ccccc|ccccc|ccccc|cccccc}
  \hline \hline 
          &       &       &       & \multicolumn{5}{c|}{Conic Cuts}             & \multicolumn{5}{c|}{CPLEX Outer Approximation}    & \multicolumn{5}{c|}{Extended Polymatroids}      & \multicolumn{6}{c}{Extended Polymatroids with Aggregated Covers} \\
    \hline
    nodes & arcs  & $\beta$  & $\Omega$ & rgap  & cuts  & nodes & time  & \# [egap]    & rgap  & cuts  & nodes & time & \# [egap]    & rgap  & cuts  & nodes & time & \# [egap]    & rgap  & cuts & covers & nodes & time  & \# [egap] \\
    \hline
	\multirow{9}{*}{10} & \multirow{9}{*}{26} & \multirow{3}{*}{0.3} & 1     & {59.2} & 13    & {413} & {4.7} & {5 } & {59.5} & 13    & {276} & {0.1} & {5 } & {59} & 11    & {56} & {0 } & {5} & {43.4} & 8     & 4     & {22} & {0} & {5} \\
          &       &       & 3     & {55.9} & 11    & {866} & {52.7} & {5} & {57.5} & 11    & {204} & {0.1} & {5 } & {56.2} & 13    & {98} & {0.1} & {5} & {40.2} & 17    & 10    & {41} & {0} & {5} \\
          &       &       & 5     & {52.8} & 9     & {363} & {16.1} & {5} & {54.2} & 9     & {164} & {0.1} & {5 } & {53.4} & 52    & {47} & {0.1} & {5} & {37.4} & 63    & 29    & {19} & {0} & {5} \\
          &       & \multirow{3}{*}{0.5} & 1     & {57.3} & 7     & {631} & {28.1} & {5 } & {57.5} & 7     & {60} & {0.1} & {5 } & {57.5} & 19    & {68} & {0.1} & {5} & {39.7} & 22    & 9     & {33} & {0} & {5} \\
          &       &       & 3     & {51.4} & 10    & {933} & {78.3} & {5} & {52.6} & 10    & {168} & {0.1} & {5 } & {52.1} & 29    & {103} & {0.1} & {5} & {33.7} & 26    & 13    & {66} & {0.1} & {5} \\
          &       &       & 5     & {48.2} & 8     & {854} & {43.5} & {5} & {49.1} & 8     & {95} & {0} & {5 } & {48.8} & 32    & {91} & {0.1} & {5} & {31.1} & 38    & 20    & {48} & {0.1} & {5} \\
          &       & \multirow{3}{*}{0.7} & 1     & {53.8} & 6     & {816} & {51.6} & {5 } & {54.1} & 6     & {171} & {0.1} & {5 } & {53.6} & 36    & {99} & {0.1} & {5} & {37.2} & 30    & 15    & {39} & {0} & {5} \\
          &       &       & 3     & {48.9} & 9     & {485} & {14.6} & {5 } & {48.8} & 9     & {337} & {0.7} & {5 } & {49.2} & 60    & {65} & {0.1} & {5} & {29.7} & 78    & 42    & {36} & {0} & {5} \\
          &       &       & 5     & {44.6} & 7     & {1432} & {128.4} & {5 } & {44.1} & 7     & {300} & {0.5} & {5 } & {44} & 37    & {161} & {0.2} & {5} & {25.2} & 51    & 25    & {71} & {0.1} & {5} \\
    \hline
    \multirow{9}{*}{20} & \multirow{9}{*}{56} & \multirow{3}{*}{0.3} & 1     & {47.4} & 67    & {55692} & {2582} & {3 [2.5]} & {48.6} & 80    & {17546} & {13.3} & {5 } & {47.6} & 577   & {6453} & {5.9} & {5} & {32} & 692   & 350   & {1021} & {3.8} & {5} \\
          &       &       & 3     & {43.5} & 60    & {111141} & {3287} & {2 [14.2]} & {44.3} & 65    & {12956} & {10.2} & {5 } & {44.1} & 682   & {11769} & {16.1} & {5} & {24.5} & 697   & 325   & {1497} & {4.7} & {5} \\
          &       &       & 5     & {47.2} & 81    & {68984} & {2803} & {3 [8.0]} & {47} & 67    & {8234} & {7.5} & {5 } & {47.5} & 749   & {7314} & {6.2} & {5} & {30.3} & 616   & 275   & {937} & {3.9} & {5} \\
          &       & \multirow{3}{*}{0.5} & 1     & {45.6} & 80    & {101183} & {3527} & {1 [12.7]} & {45.6} & 54    & {159690} & {1251.3} & {3 [7.5]} & {45.3} & 310   & {18110} & {19.2} & {5} & {25.7} & 250   & 136   & {1541} & {5.0} & {5} \\
          &       &       & 3     & {44.2} & 62    & {141726} & {3600} & {0 [19.9]} & {44.1} & 58    & {196375} & {2101.1} & {3 [16.8]} & {44.2} & 259   & {23994} & {24.1} & {5} & {28.3} & 373   & 174   & {2313} & {8.6} & {5} \\
          &       &       & 5     & {49.8} & 72    & {107933} & {3600} & {0 [28]} & {50} & 46    & {186240} & {1874.2} & {3 [16.3]} & {49.8} & 753   & {20107} & {20.8} & {5} & {32.9} & 850   & 458   & {2579} & {9.47} & {5} \\
          &       & \multirow{3}{*}{0.7} & 1     & {58.3} & 56    & {124710} & {3600} & {0 [43.3]} & {58.6} & 71    & {156478} & {1441.6} & {3 [39.2]} & {58.8} & 725   & {20302} & {22.4} & {5} & {40.7} & 482   & 274   & {1875} & {5.84} & {5} \\
          &       &       & 3     & {47.1} & 38    & {104174} & {3600} & {0 [16.3]} & {48.2} & 83    & {24724} & {15.2} & {5 } & {48} & 569   & {11657} & {15.07} & {5} & {28.4} & 606   & 305   & {1654} & {5.8} & {5} \\
          &       &       & 5     & {49.6} & 80    & {126946} & {3600} & {0 [19.9]} & {50.9} & 40    & {11811} & {10.1} & {5 } & {50.7} & 640   & {13289} & {18.2} & {5} & {33.3} & 818   & 342   & {997} & {4.1} & {5} \\
    \hline
    \multirow{9}{*}{30} & \multirow{9}{*}{134} & \multirow{3}{*}{0.3} & 1     & {68.9} & 109   & {106396} & {3600} & {0 [43.7]} & {68.6} & 188   & {578013} & {3496.2} & {1 [9.7]} & {68.3} & 3346  & {39550} & {42.7} & {5} & {49.6} & 3180  & 1652  & {4101} & {21.7} & {5} \\
          &       &       & 3     & {70.1} & 124   & {190983} & {3600} & {0 [50.5]} & {70.8} & 157   & {483580} & {3284.1} & {2 [4.5]} & {70} & 3220  & {43285} & {49.7} & {5} & {50.1} & 3068  & 1353  & {4638} & {26.3} & {5} \\
          &       &       & 5     & {65.3} & 141   & {174411} & {3600} & {0 [26.7]} & {67.1} & 142   & {532954} & {3600} & {0 [25.1]} & {66.3} & 2265  & {44096} & {39.3} & {5} & {48.4} & 2292  & 1373  & {3855} & {19.7} & {5} \\
          &       & \multirow{3}{*}{0.5} & 1     & {89.7} & 146   & {73266} & {3600} & {0 [81]} & {89} & 154   & {591382} & {3600} & {0 [50.6]} & {89.3} & 3145  & {36896} & {84.9} & {5} & {74} & 2470  & 1013  & {3192} & {15.7} & {5} \\
          &       &       & 3     & {68.2} & 118   & {180719} & {3600} & {0 [38.1]} & {68.9} & 144   & {637194} & {3600} & {0 [24.8]} & {69} & 2767  & {49934} & {80.3} & {5} & {50.1} & 3152  & 1526  & {5215} & {33.9} & {5} \\
          &       &       & 5     & {65.5} & 136   & {184461} & {3600} & {0 [19.5]} & {66.2} & 121   & {581235} & {3600} & {0 [31.1]} & {66.2} & 2531  & {47512} & {77.2} & {5} & {49.3} & 1535  & 858   & {4827} & {29.9} & {5} \\
          &       & \multirow{3}{*}{0.7} & 1     & {72.3} & 124   & {110874} & {3600} & {0 [45.7]} & {72.4} & 145   & {628966} & {3600} & {0 [44.3]} & {72.1} & 2243  & {42543} & {66.7} & {5} & {52.1} & 2800  & 1204  & {3977} & {20.9} & {5} \\
          &       &       & 3     & {63.3} & 97    & {169244} & {3600} & {0 [24.5]} & {63} & 110   & {514648} & {3600} & {0 [15.5]} & {63.1} & 2240  & {42657} & {69.1} & {5} & {48} & 2933  & 1302  & {4022} & {21.8} & {5} \\
          &       &       & 5     & {87.8} & 110   & {186748} & {3600} & {0 [70.5]} & {87.7} & 187   & {513009} & {3600} & {0 [33.4]} & {88.1} & 2820  & {44332} & {91.8} & {5} & {48.9} & 1627  & 661   & {4691} & {29.6} & {5} \\
    \hline
    \multicolumn{4}{c|}{\textbf{avg}}   & \textbf{57.6}  & \textbf{66}    & \textbf{86162} & \textbf{2334}  &    & \textbf{58.1}  & \textbf{74}    & \textbf{216178} & \textbf{1434}  & & \textbf{57.9}     & \textbf{1116}  & \textbf{19429} & \textbf{28}    &  & \textbf{39.4}     & \textbf{1066}  & \textbf{509}   & \textbf{1974}  & \textbf{10}    &  \\
    \multicolumn{4}{c|}{\textbf{stdev}}     & \textbf{12.6}  &   \textbf{49}    &     \textbf{70804}  &    \textbf{1666}   &       & \textbf{12.4}  &    \textbf{61}   &   \textbf{258096}    &    \textbf{1639}   &       & \textbf{12.4}     &    \textbf{1215}   &   \textbf{18773}    &    \textbf{31}   &       & \textbf{11.3}     &  \textbf{1163}     &   \textbf{554}    &  \textbf{1844}     &    \textbf{11}   &  \\
  \hline \hline 
    \end{tabular}}%
  \label{tab:networkgeneral}%
\end{sidewaystable}%

\subsection{Large scale instances}

Finally we test how the proposed approach scales for large scale network design instances derived from 1990 U.S. Census described in \citet{daskin-book}. The data set is available at \texttt{\url{http://umich.edu/~msdaskin/files/sitation-class-2013.zip}}. We use the two largest data sets consisting of 88 and 150 cities in the United Stated. The arc costs are set proportional to the great circle distances between the cities as available in the data. The remaining parameters are generated as described in Section~\ref{sec:data} with a general covariance matrix.
These instances are run on a six-core Intel Xeon 3.6GHz computer with 32GB main memory with a two hour time limit. The results summarized in 
Table~\ref{tab:daskin} show that none of the instances are solved well by default CPLEX. On the other hand, with the addition of the 
cuts all of the 88-city instances and 26 of the 45 150-city instances are solved to optimality. For the remaining 19 instances the average gap between the lower and upper bounds at termination (egap) is 65.6\% lower 
compared to the default CPLEX.

\begin{table}[htbp]
  \centering
  \caption{\small Network design: Computations on Daskin Data Set.}
    \resizebox{\textwidth}{!}{%
    \renewcommand{\arraystretch}{1.2}
    \begin{tabular}{ccc|ccccc|cccccc}
    \hline \hline
          &       &       & \multicolumn{5}{c|}{CPLEX}            & \multicolumn{6}{c}{Polymatroids and Aggregated Covers} \\
    \hline
    nodes & arcs  & $\Omega$ & rgap  & cuts  & nodes & time  & \# [egap] & rgap  & cuts & covers & nodes & time  & \# [egap] \\
    \hline
    \multirow{9}{*}{88} & \multirow{3}{*}{277} & 1     & 86.2  & 33    & 55694 & 7200  & 0 [50.7] & 28.5  & 5115  & 1639  & 5691  & 1391  & 5 \\
          &       & 3     & 82.8  & 30    & 52219 & 7200  & 0 [48.3] & 26.2  & 5702  & 1912  & 5828  & 1554  & 5 \\
          &       & 5     & 80.1  & 26    & 57557 & 7200  & 0 [43.2] & 25.4  & 7439  & 1522  & 6021  & 1573  & 5 \\
          & \multirow{3}{*}{478} & 1     & 81.2  & 27    & 90284 & 7200  & 0 [41.4] & 35.5  & 12108 & 4559  & 8782  & 2884  & 5 \\
          &       & 3     & 80.8  & 30    & 64326 & 7200  & 0 [38.4] & 30.2  & 13005 & 4921  & 8426  & 2931  & 5 \\
          &       & 5     & 77.1  & 26    & 67357 & 7200  & 0 [41.2] & 28.4  & 13881 & 5022  & 9021  & 2695  & 5 \\
          & \multirow{3}{*}{890} & 1     & 72.6  & 23    & 71154 & 7200  & 0 [38.4] & 43.7  & 17862 & 6311  & 15933 & 5021  & 5 \\
          &       & 3     & 70.2  & 21    & 82498 & 7200  & 0 [35.1] & 41.9  & 17183 & 5993  & 15381 & 5392  & 5 \\
          &       & 5     & 68.8  & 26    & 80676 & 7200  & 0 [31.3] & 41.2  & 18336 & 6334  & 16017 & 5211  & 5 \\
    \hline
    \multirow{9}{*}{150} & \multirow{3}{*}{612} & 1     & 87.3  & 27    & 60795 & 7200  & 0 [39.8] & 40.1  & 15512 & 5512  & 13024 & 3706  & 5 \\
          &       & 3     & 83.7  & 29    & 81748 & 7200  & 0 [36.3] & 39.6  & 16067 & 5701  & 13922 & 3919  & 5 \\
          &       & 5     & 81.1  & 28    & 84404 & 7200  & 0 [34.4] & 36.2  & 15482 & 5521  & 13433 & 4092  & 5 \\
          & \multirow{3}{*}{1062} & 1     & 84.2  & 22    & 92482 & 7200  & 0 [39.2] & 49.1  & 26823 & 7812  & 31093 & 6126  & 4 [9.1] \\
          &       & 3     & 81.6  & 23    & 94588 & 7200  & 0 [36.4] & 46.6  & 30067 & 8108  & 33911 & 6319  & 4 [8.4] \\
          &       & 5     & 80.1  & 27    & 150708 & 7200  & 0 [34.5] & 45.8  & 31482 & 8791  & 34224 & 6892  & 3 [10.9] \\
          & \multirow{3}{*}{1945} & 1     & 87.3  & 22    & 510154 & 7200  & 0 [42.3] & 58.4  & 48513 & 12819 & 58913 & 7200  & 0 [21.6] \\
          &       & 3     & 89.1  & 26    & 824985 & 7200  & 0 [46.8] & 60.8  & 50226 & 13088 & 61891 & 7200  & 0 [19.8] \\
          &       & 5     & 87.1  & 21    & 880676 & 7200  & 0 [47.3] & 57.7  & 48943 & 12947 & 60551 & 7200  & 0 [25.1] \\
    \hline
    \multicolumn{3}{c|}{\textbf{avg}} & \multicolumn{1}{c}{\textbf{81.2}} & \multicolumn{1}{c}{\textbf{25.9}} & \multicolumn{1}{c}{\textbf{189017}} & \multicolumn{1}{c}{\textbf{7200}} &       & \multicolumn{1}{c}{\textbf{40.8}} & \multicolumn{1}{c}{\textbf{21874.8}} & \multicolumn{1}{c}{\textbf{6584}} & \multicolumn{1}{c}{\textbf{22892.3}} & \multicolumn{1}{c}{\textbf{4517}} &  \\
    \multicolumn{3}{c|}{\textbf{stdev}} & \multicolumn{1}{c}{\textbf{5.9}} & \multicolumn{1}{c}{\textbf{3.4}} & \multicolumn{1}{c}{\textbf{262949}} & \multicolumn{1}{c}{\textbf{0}} &       & \multicolumn{1}{c}{\textbf{10.9}} & \multicolumn{1}{c}{\textbf{14520.7}} & \multicolumn{1}{c}{\textbf{3557.2}} & \multicolumn{1}{c}{\textbf{19471.8}} & \multicolumn{1}{c}{\textbf{2050.7}} &  \\
    \hline \hline
    \end{tabular}}%
\label{tab:daskin}
\end{table}%

\section{Conclusion}
In this paper we consider single-commodity network design with probabilistic arc capacities. The problem is modeled with an exponential number of probabilistic capacity constraints. We study the independent and correlated cases separately as they lead to significantly different constraint sets and separation problems. In both cases, we observe that exploiting the underlying submodularity and supermodularity arising with the probabilistic constraints provides significant advantages over the classical approaches.

We have not considered joint-probabilistic constraints across all cuts of the network, which leads to a non-convex feasible set even for the continuous relaxation of the problem.
We have not considered multiple commodities either. For the multicommodity case, projecting out flow variables requires the use of metric inequalities on the design variables. We leave probabilistic metric constraints as a future research topic.

\section*{Acknowledgement}
This research has been supported, in part, by grant FA9550-10-1-0168 from the Office of the Assistant Secretary of Defense for Research and Engineering.

\bibliographystyle{abbrvnat}
\bibliography{NetworkDesign}
\begin{appendices}
	\titleformat{\section}{\large\bfseries}{\appendixname~\thesection .}{0.5em}{}
	\section{Simulation Details}
	\label{app:1}

	\subsection{Network data}
	\label{app:data}

For the example in Section~\ref{sec:intro}, we use a
six-node acyclic complete network with demand 230 $s$ to $t$ is 230.
Mean capacities, variances and costs for the arcs are presented in Table \ref{tab:1}.

\begin{table}[h!]
\caption{Candidate arcs and data.}
\centering
	\begin{tabular}{c|c|c|c|c|c}
		\hline \hline
		edge\# & from & to & mean capacity & variance & cost\\
		\hline
		1 & s & 1 & 81 & 16 & 14\\
		2 & s& 2& 90 & 676 & 42\\
		3 & s& 3& 12 & 4 & 91\\
		4 & s& 4& 91 & 289 & 79\\
		5 & s& t& 63 & 9 & 95\\
		6 & 1& 2& 9 & 4 & 65\\
		7 & 1& 3& 27 & 25 & 3\\
		8 & 1& 4& 54 &36 & 84\\
		9 & 1 & t& 95 & 256 & 93\\
		10 & 2 & 3 & 96 & 169 & 67\\
		11 & 2 & 4 & 15 & 1 & 75\\
		12 & 2 & t & 97 & 64 & 74\\
		13 & 3 & 4 & 95 & 16 & 39\\
		14 & 3 & t & 48 & 9 & 65\\
		15 & 4 & t & 80 & 49 & 17\\
	\hline \hline
	\end{tabular}
\label{tab:1}
\end{table}

\newpage

	\subsection{Minimum cost network configurations for different service levels}
	\label{app:conf}
	
	\vspace{5mm}
	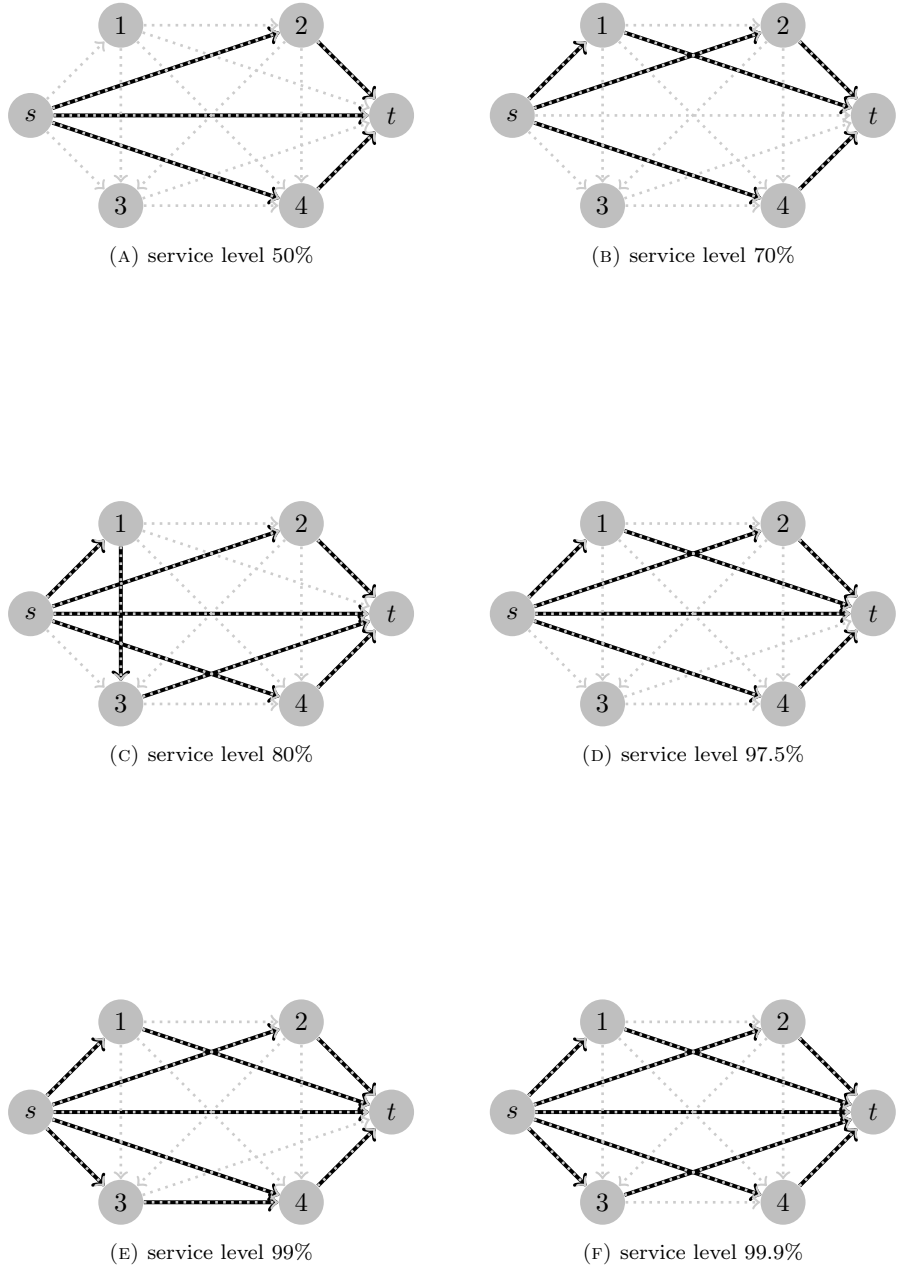
\begin{figure}[h!]
	\centering
	\begin{subfigure}{0.5\textwidth}
	\centering
	\begin{tikzpicture}[scale=1.2, auto,swap]
	    \foreach \pos/\name in {{(2,1)/1}, {(4,1)/2},
	                            {(1,0)/s}, {(5,0)/t}, {(2,-1)/3}, {(4,-1)/4}}
	        \node[vertex] (\name) at \pos {$\name$};
	    \foreach \source/ \dest in {s/2,s/4,s/t,2/t,4/t}
	        \path[selected edge] (\source)--(\dest);
	
	    \foreach \source/ \dest in {s/1,s/2,s/3, s/4,s/t,1/2,1/3,1/4,1/t,2/3,2/4,2/t,3/4,3/t,4/t}
	        \path[ignored edge] (\source)--(\dest);
	\end{tikzpicture}
	\caption{service level 50\%}
	\label{fig:1a}
	\end{subfigure}%
	~
	\begin{subfigure}{0.5\textwidth}
	\centering
	\begin{tikzpicture}[scale=1.2, auto,swap]
	    \foreach \pos/\name in {{(2,1)/1}, {(4,1)/2},
	                            {(1,0)/s}, {(5,0)/t}, {(2,-1)/3}, {(4,-1)/4}}
	        \node[vertex] (\name) at \pos {$\name$};
	    \foreach \source/ \dest in {s/1,s/2,s/4,1/t,2/t,4/t}
	        \path[selected edge] (\source)--(\dest);
	
	    \foreach \source/ \dest in {s/1,s/2,s/3, s/4,s/t,1/2,1/3,1/4,1/t,2/3,2/4,2/t,3/4,3/t,4/t}
	        \path[ignored edge] (\source)--(\dest);
	\end{tikzpicture}
	\caption{service level 70\%}
	\label{fig:1b}
	\end{subfigure}
	\newline \vspace{30mm}
	\begin{subfigure}{0.5\textwidth}
	\centering
	\begin{tikzpicture}[scale=1.2, auto,swap]
	    \foreach \pos/\name in {{(2,1)/1}, {(4,1)/2},
	                            {(1,0)/s}, {(5,0)/t}, {(2,-1)/3}, {(4,-1)/4}}
	        \node[vertex] (\name) at \pos {$\name$};
	    \foreach \source/ \dest in {s/1,s/2,s/4,s/t,1/3,2/t,3/t,4/t}
	        \path[selected edge] (\source)--(\dest);
	
	    \foreach \source/ \dest in {s/1,s/2,s/3, s/4,s/t,1/2,1/3,1/4,1/t,2/3,2/4,2/t,3/4,3/t,4/t}
	        \path[ignored edge] (\source)--(\dest);
	\end{tikzpicture}
	\caption{service level 80\%}
	\label{fig:1c}
	\end{subfigure}
	~
	\begin{subfigure}{0.5\textwidth}
	\centering
	\begin{tikzpicture}[scale=1.2, auto,swap]
	    \foreach \pos/\name in {{(2,1)/1}, {(4,1)/2},
	                            {(1,0)/s}, {(5,0)/t}, {(2,-1)/3}, {(4,-1)/4}}
	        \node[vertex] (\name) at \pos {$\name$};
	    \foreach \source/ \dest in {s/1,s/2,s/4,s/t,1/t,2/t,4/t}
	        \path[selected edge] (\source)--(\dest);
	
	    \foreach \source/ \dest in {s/1,s/2,s/3, s/4,s/t,1/2,1/3,1/4,1/t,2/3,2/4,2/t,3/4,3/t,4/t}
	        \path[ignored edge] (\source)--(\dest);
	\end{tikzpicture}
	\caption{service level 97.5\%}
	\label{fig:1d}
	\end{subfigure}
	\newline \vspace{30mm}
	\begin{subfigure}{0.5\textwidth}
	\centering
	\begin{tikzpicture}[scale=1.2, auto,swap]
	    \foreach \pos/\name in {{(2,1)/1}, {(4,1)/2},
	                            {(1,0)/s}, {(5,0)/t}, {(2,-1)/3}, {(4,-1)/4}}
	        \node[vertex] (\name) at \pos {$\name$};
	    \foreach \source/ \dest in {s/1,s/2,s/3,s/4,s/t,1/t,2/t,3/4,4/t}
	        \path[selected edge] (\source)--(\dest);
	
	    \foreach \source/ \dest in {s/1,s/2,s/3, s/4,s/t,1/2,1/3,1/4,1/t,2/3,2/4,2/t,3/4,3/t,4/t}
	        \path[ignored edge] (\source)--(\dest);
	\end{tikzpicture}
	\caption{service level 99\%}
	\label{fig:1e}
	\end{subfigure}%
	~
	\begin{subfigure}{0.5\textwidth}
	\centering
	\begin{tikzpicture}[scale=1.2, auto,swap]
	    \foreach \pos/\name in {{(2,1)/1}, {(4,1)/2},
	                            {(1,0)/s}, {(5,0)/t}, {(2,-1)/3}, {(4,-1)/4}}
	        \node[vertex] (\name) at \pos {$\name$};
	    \foreach \source/ \dest in {s/1,s/2,s/3,s/4,s/t,1/t,2/t,3/t,4/t}
	        \path[selected edge] (\source)--(\dest);
	
	    \foreach \source/ \dest in {s/1,s/2,s/3, s/4,s/t,1/2,1/3,1/4,1/t,2/3,2/4,2/t,3/4,3/t,4/t}
	        \path[ignored edge] (\source)--(\dest);
	\end{tikzpicture}
	\caption{service level 99.9\%}
	\label{fig:1f}
	\end{subfigure}
	\caption{Network configurations corresponding to different service levels.}
\end{figure}

\pagebreak

	\subsection{Distribution of minimum cut capacities for the simulated instances}
	\label{app:dist}
	\begin{figure}[H]
	\centering
	\begin{subfigure}{0.5\textwidth}
	\centering
	\includegraphics[scale = 0.15]{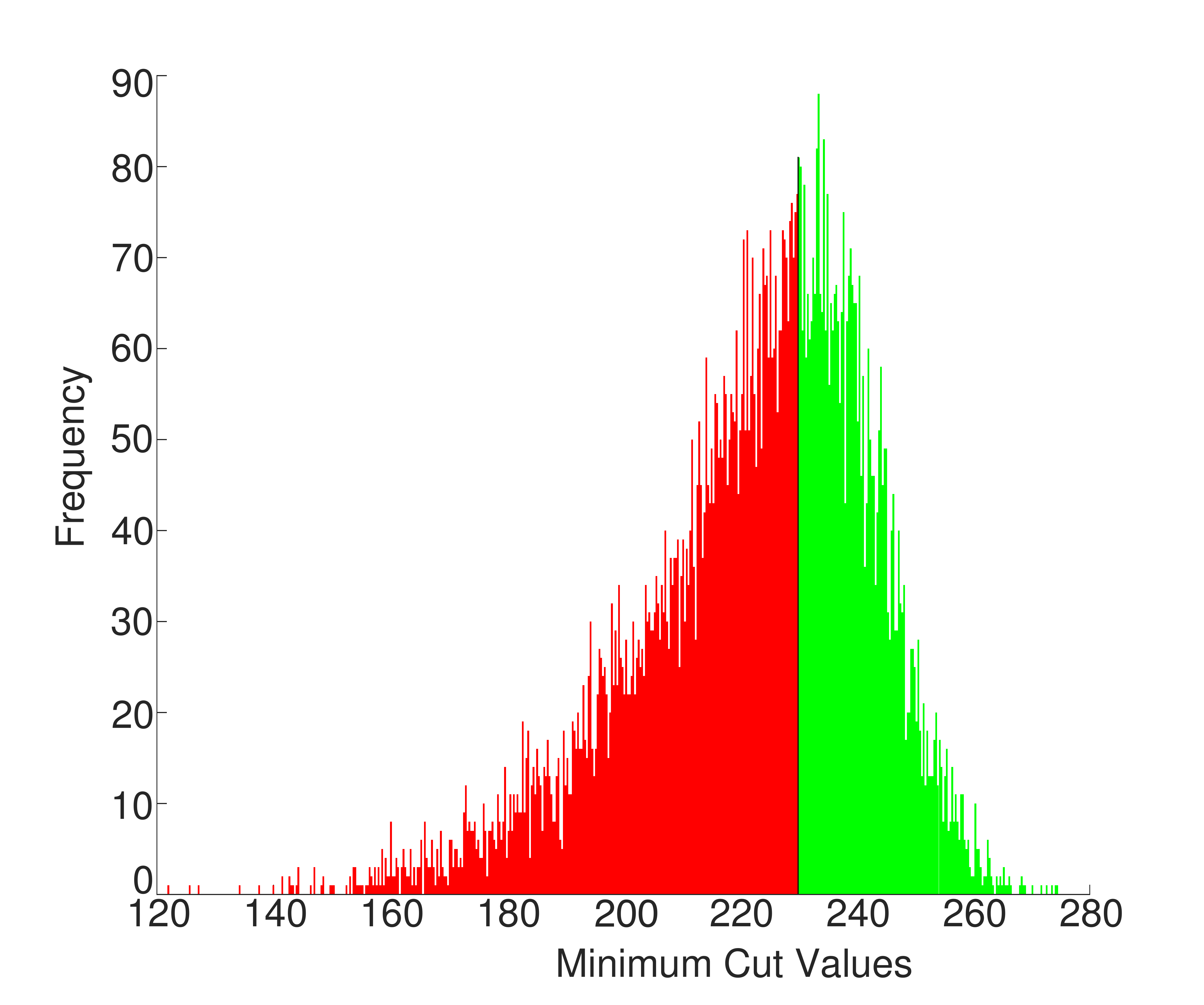}
	\caption{$1-\epsilon = 50\%$}
	\label{fig:2a}
	\end{subfigure}%
	~
	\begin{subfigure}{0.5\textwidth}
	\centering
	\includegraphics[scale = 0.15]{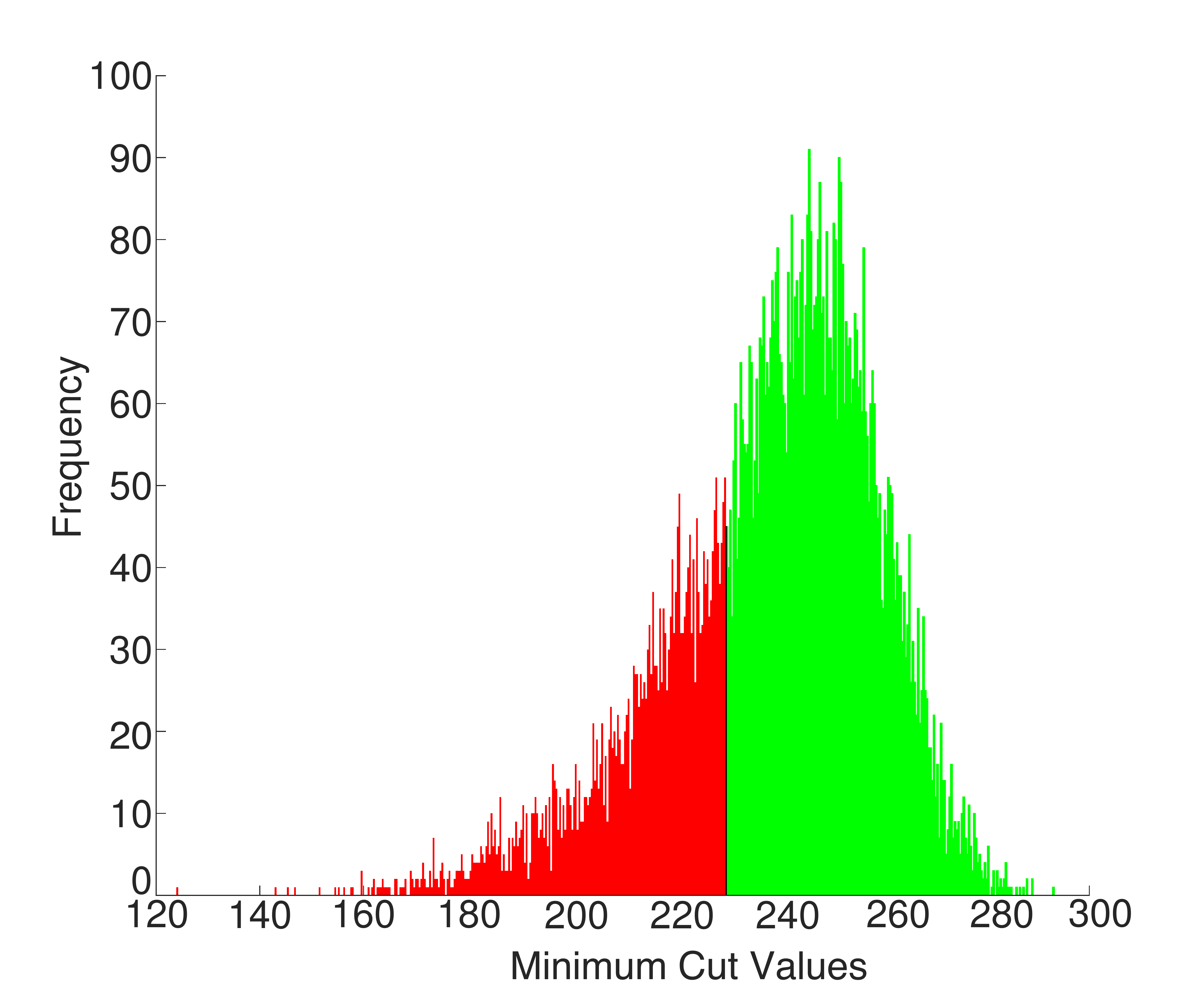}
	\caption{$1-\epsilon = 70\%$}
	\label{fig:2b}
	\end{subfigure}
	\newline \vspace{8mm}
	\begin{subfigure}{0.5\textwidth}
	\centering
	\includegraphics[scale = 0.15]{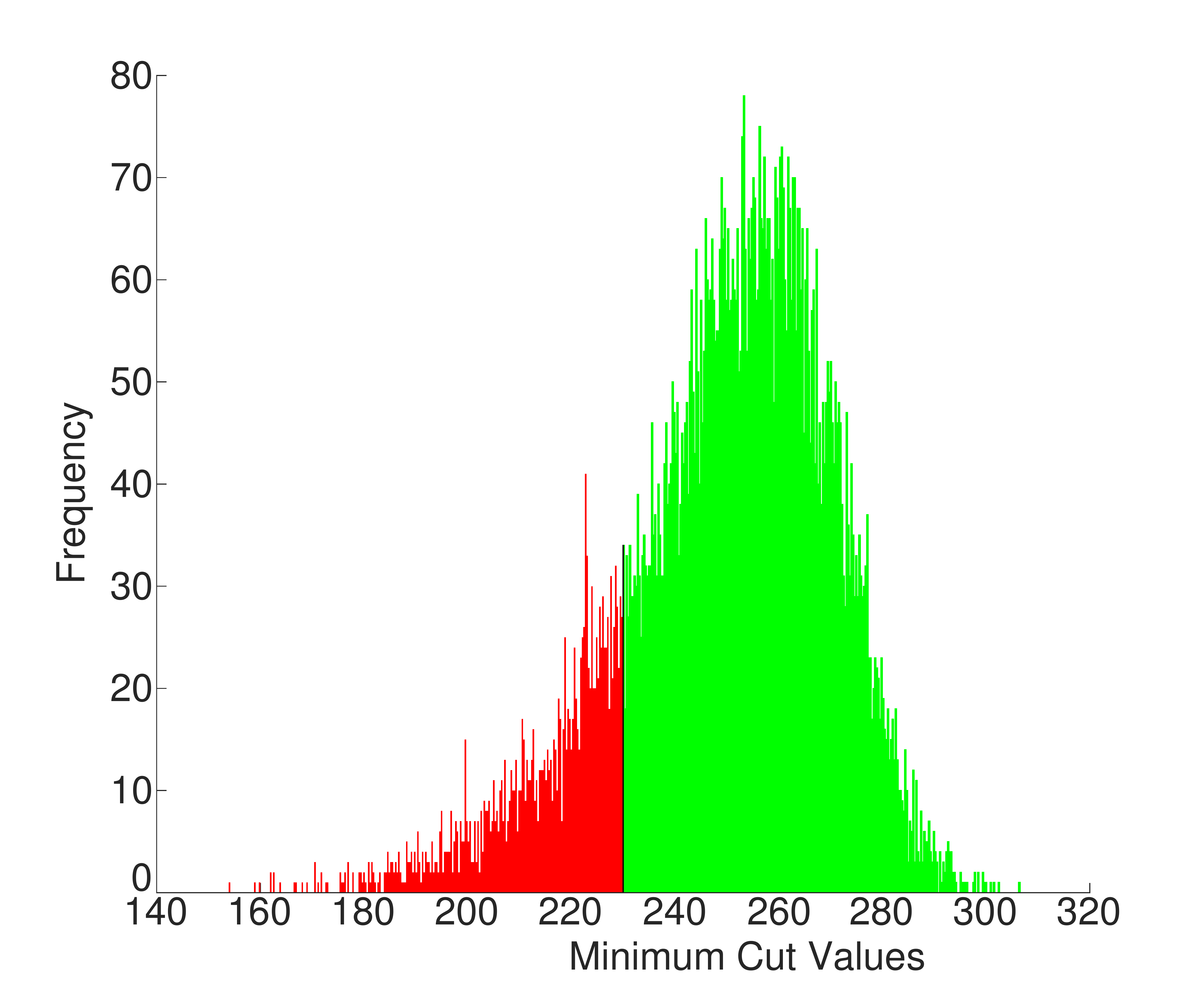}
	\caption{$1-\epsilon = 80\%$}
	\label{fig:2c}
	\end{subfigure}%
	~
	\begin{subfigure}{0.5\textwidth}
	\centering
	\includegraphics[scale = 0.15]{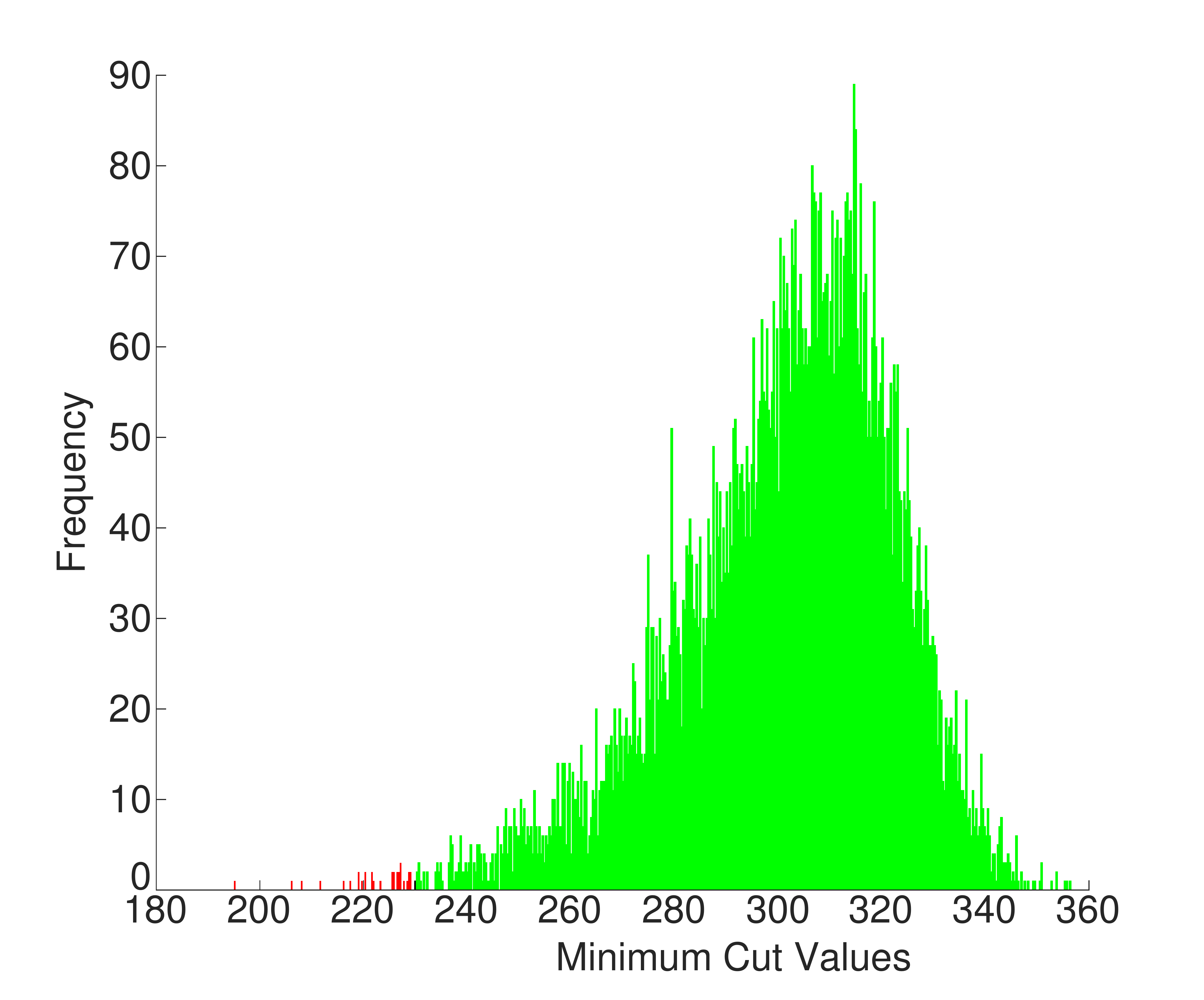}
	\caption{$1-\epsilon = 97.5\%$}
	\label{fig:2d}	\end{subfigure}
	\newline \vspace{8mm}
	\begin{subfigure}{0.5\textwidth}
	\centering
	\includegraphics[scale = 0.15]{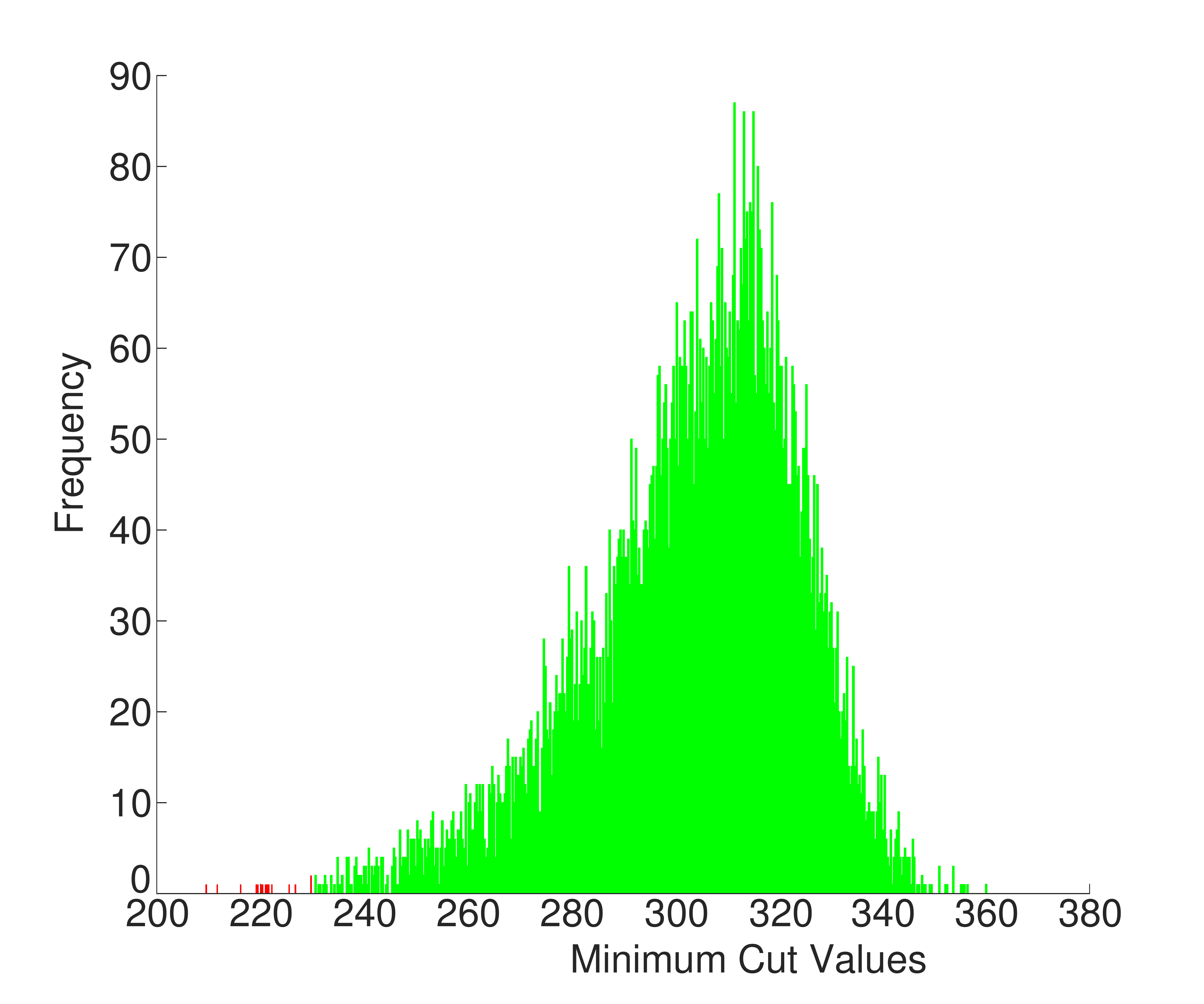}
	\caption{$1-\epsilon = 99\%$}
	\label{fig:2e}
	\end{subfigure}%
	~
	\begin{subfigure}{0.5\textwidth}
	\centering
	\includegraphics[scale = 0.15]{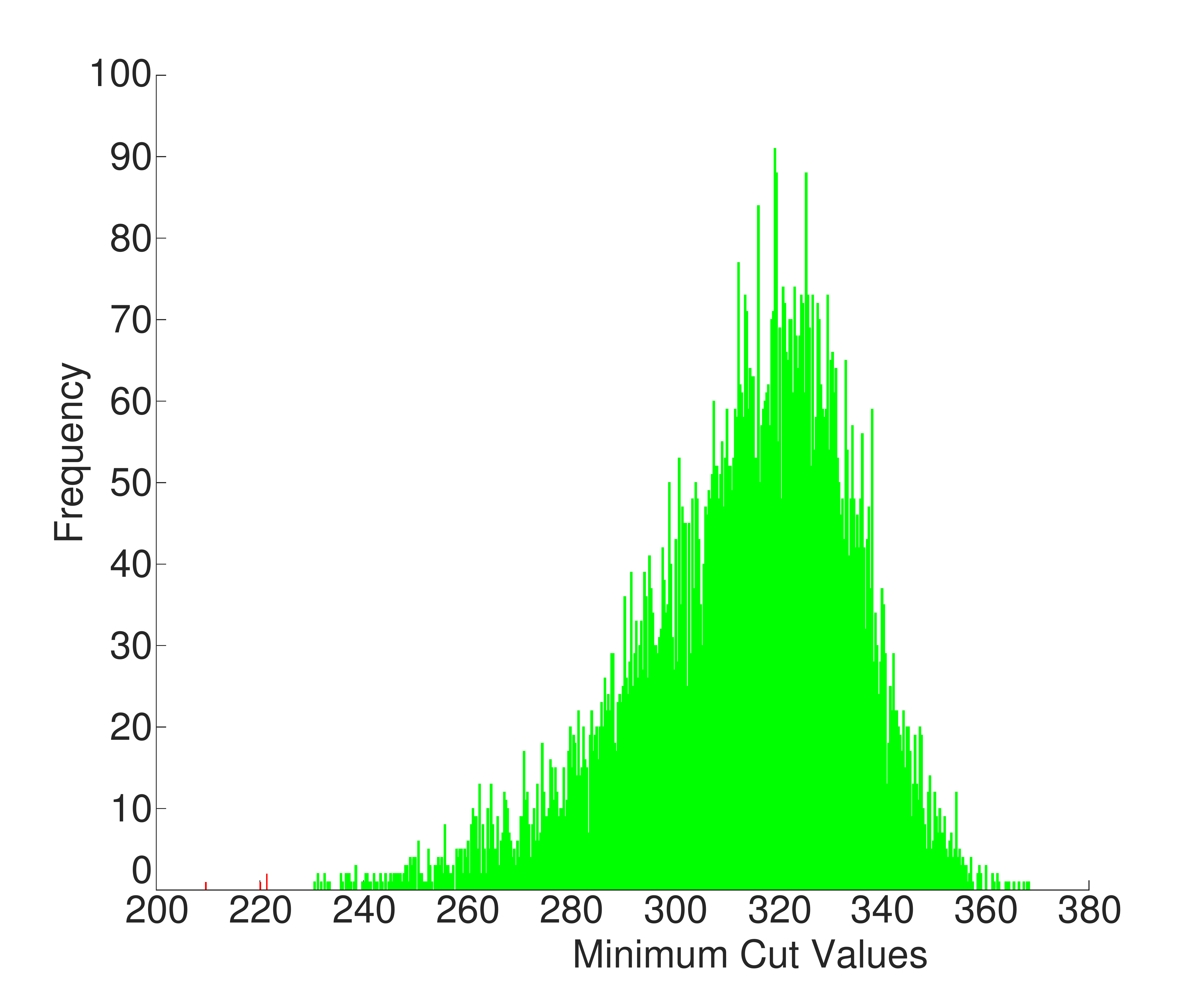}
	\caption{$1-\epsilon = 99.9\%$}
	\label{fig:2f}
	\end{subfigure}
	\caption{\scriptsize{Minimum cut capacities for simulated instances.}}
	\end{figure}
\end{appendices}

\end{document}